\documentclass{article}
\usepackage[utf8]{inputenc}
\usepackage{latexsym}
\usepackage{amsfonts}
\usepackage{amsthm}
\usepackage{amsmath}
\usepackage{graphicx}
\usepackage{parskip}

\newtheorem{theorem}{Theorem}[section]

\newtheorem{lemma}[theorem]{Lemma}

\newtheorem{definition}[theorem]{Definition}
\theoremstyle{remark}
\newtheorem{remark}[theorem]{Remark}
\newtheorem{example}[theorem]{Example}

\def\Abar{{\bar{A}}}

\def\deltap{{\delta}'}
\def\delp{{\partial}'}
\def\Q{{\mathbb Q}}

\def\db{{\bar{\d}}}
\def\emptyb{\overline{\emptyset}}

\def\Bar{\overline{\phantom{x}}}

\def\chib{{\bar{\chi}}}

\def\starb{{\bar{*}}}
\def\Bbar{{\bar{B}}}
\def\Ebar{{\bar{E}}}
\def\Dbar{{\overline{D}}}
\def\xbar{{\bar{x}}}

\def\mbar{{\overline{m}}}
\def\Ibar{{\bar{I}}}
\def\Cbar{{\bar{C}}}
\def\half{{\frac12}}

\def\d{\partial}
\def\id{{\rm id}}

\def\idb{{\overline{\rm id}}}

\def\vbar{\bar{v}}

\def\wbar{\bar{w}}
\def\dbar{\bar{\d}}
\def\taubar{{\overline{\tau}}}
\def\Vbar{\overline{V}}
\def\Wbar{\overline{W}}
\def\deltabar{\bar{\delta}}

\def\ubar{\bar{u}}

\def\eu{{\bf e}_u}
\def\Tu{T_{u}}
\def\Tup{T'_{u}}
\def\tu{T'_{u}}
\def\Deltab{{\bar{\Delta}}}
\def\bs{\backslash}

\def\Z{\mathbb{Z}}
\def\R{\mathbb{R}}
\def\N{\mathbb{N}}

\def\T{\mathbb{T}}
\def\binomial#1#2{\begin{pmatrix}#1\\#2\end{pmatrix}}

\title{Quantitative towers in finite difference calculus approximating the continuum\footnote{To appear in {\it Quarterly Journal of Mathematics}}}
\author{R.~Lawrence\footnote{Einstein Institute of Mathematics, Hebrew University of Jerusalem},\  N.~Ranade\footnote{Cold Spring Harbor Laboratory, NY}  \ and D.~Sullivan\footnote{CUNY Graduate Center NY \& SUNY, Stony Brook NY}}
\date{\it In memory of Sir Michael Atiyah}

\begin{document}
\maketitle

\textbf{Abstract}
Multivector fields and differential forms at the continuum level have respectively
 two commutative associative products, a third  composition product between them  and various operators like $\d$, $d$ and $*$ which are used to describe many nonlinear problems. The point of this paper is to construct consistent direct and inverse systems of finite dimensional approximations to these structures and to calculate combinatorially how these finite dimensional models differ from their continuum idealizations.  In a Euclidean background there is an explicit  answer which is natural \textbf{statistically}.

 \section{Introduction}
At the continuum level there is a rich nonlinear structure (see Background Appendix) with more symmetry than is possible for any system of finite dimensional approximations.

  The discrete vector calculus discussed here is  based on combinatorial topology with an  unexpected Grassmann algebra fit. One  will  have at each scale
  two finite dimensional graded commutative associative algebras: one product on chains and one product on cochains.  There  is  a  third  contraction  product  between  chains  and cochains.   The  chains  have  their  differential $\d$ of  degree ${-}1$  and  the  cochains have  their  differential  $\delta$ of  degree  +1.   Neither  $\d$ nor $\delta$ is  a  derivation  of  its product.  Also
  $\d$ is further away by one scale order of magnitude from being a derivation than $\delta$ is.

  In periodic Euclidean-space the new point \cite{S}
  is to consider $2{^n}$ overlapping scale $2h$-cubical grids where $h=2{^{-i-1}}$ at level $i$ to build finite-dimensional spaces of combinatorial chains and cochains.  This particular choice allows discrete vector and covector Grassmann products and their contraction product.  The boundary  operator
  $\d$ on  chains  and  the  coboundary  operator  $\delta$ on  cochains are precisely related by the Poincar\'e duality “star operator".

The operators do not “derive” their product structures as they do in the continuum analogue, but their higher order deviations from such, called “infinitesimal cumulants”, have a noteworthy $h$-divisibility structure  developed here.

   Namely, one finds for each level of discretization $i$, the  finite-dimensional algebras indicated above, after extending scalars by power series in the formal  scale parameter $h$,  with their operators $\delta$ for cochains and $\d$  for chains, satisfy $h$ divisibility properties that  make them respectively  into consistent inverse and direct systems of  binary QFT algebras. See {\S\S6,7} below.  Also similar estimates control the \textbf{traditional cumulants} of the linear mappings respecting boundary operators between the commutative algebras at different scales.

   We will apply to each finite dimensional algebra with differential, a general  procedure in order to resolve the broken symmetry of continuum algebra.

   The procedure involves  higher-order brackets which are Taylor coefficients of coderivations. These are described in \S2 and \cite{M}. The details are  new for the $\partial$ case  and for the $\delta$ case in that the structure  resonates with Jae-Suk Park's algebraic background  and computational prescription (2018) for perturbative QFT \cite{P}  referred to as binary QFT algebras. These are described in terms of divisibility of the  brackets above  by appropriate powers of the scale variable. This binary QFT formulation at the finite rank algebra level extends the theoretical physicist's Batalin-Vilkovisky formalism that requires the continuum.

 \medskip\begin{definition}  If a graded  commutative associative algebra over the formal power series in  a formal variable $h$ has a differential, this  determines  a structure   consisting of  all the brackets  referred to above (which turn out to be expressible in terms of   commutators of the  differential  with iterated  products; see example 2.5). If these all vanish this is called a classical differential algebra. If the $(k+1)$-bracket is \textbf{divisible} by $h^k$ for all $k\in\N$, this is called a \textbf{ binary QFT algebra}.
 \end{definition}

 We put ourselves in the context of this definition by extending the scalars of the  chains and cochains in the construction by the formal power series in $h$.  We add  relations to this free module  making the differences divisible by $h$, these differences being defined using the shifts by $h$  in the coordinate directions of the lattice.   More details are in \S\S6,7. The proofs fill  \S\S2,3,4,5,6,7.

\begin{theorem}
    \begin{itemize}
    \item[]
    \item[A)] The  sequence of bracket  operators of $\d$ on lattice chains  referred to above and described in \S2   are divisible by the following powers of $h$:   $1,1,2,3,\ldots$.   This means by Definition 1.1, that the operator  $\partial$ on the chains over formal power series in $h$ with the difference quotients added, defines a binary QFT algebra.

    \item[B)]  There are natural linear mappings from a coarse scale to a finer scale   respecting the $\partial$ operators and  there are  canonical structure-preserving  mappings  between  the  infinitesimal cumulant bracket structures defined at each scale. The Taylor components of these coalgebra mappings are the usual  cumulants of statistics and are divisible by  appropriate powers of $h$.
    \end{itemize}
\end{theorem}

\begin{theorem}
    \begin{itemize}
    \item[]
   \item[A)] The  sequence of bracket operators of $\delta$ on lattice cochains have scale orders $1,2,3,\ldots$ in powers of the scale $h$. This means by Definition 1.1 that the operator $h^{-1}\delta$ on the cochains  with scalars extended to the formal power series in $h$ with difference quotients added, defines a binary QFT algebra.

     \item[B)] There are natural linear mappings from a  fine  scale to a  coarser  scale   respecting the  $\delta$ operators and  there are  canonical structure-preserving  mappings  between  the  infinitesimal cumulant bracket structures defined at each scale. The Taylor components of these coalgebra mappings are the usual  cumulants of statistics and are divisible by  appropriate powers of $h$.
    \end{itemize}
\end{theorem}

      \bigskip \textbf{Takeaway:} The calculations  below will prove these theorems and  show that as the grid size tends to zero, the  infinitesimal cumulants of the   $h^{-1}\d$ structure on chains with respect to the product, tends to  zero starting  at the 3-bracket. The $h^{-1}\d$ structure  converges  to the $\d$  differential geometric algebra discussed in the Background  Appendix, where the  commutator with $\d$ there includes the Lie  bracket of vector fields.

     For cochains the quadratic term also tends to zero,  with the  $h^{-1}\delta$ operator  becoming  approximately a derivation of the exterior  product on cochains, and the  entire structure approaches the differential algebra of differential forms, also  with  the  features as discussed in the Background Appendix.

    Finally, these algebras of chains and cochains make up consistent direct and inverse  systems with explicit structure-preserving mappings. The Taylor coefficients of the  mappings between scales  are the cumulants of random variables in statistics  and these satisfy appropriate divisibility by powers of $h$.

    The entire package in a Euclidean background might be thought of as a universal explicit prescription for treating the  interesting  closure problem in finite dimensional approximations to  nonlinear problems. It might substitute for the  ``effective action" approximations of theoretical physics when the action principle is not apparent or not paramount.

 \textbf{The layout of the paper is as follows.}
\S2 recalls for the reader the multi-bracket structure defined by the cumulant bijection and builds most of the combinatorics needed for the proofs. The geometry of the lattice beginning in three-dimensions is explained in \S{3}. The divisibility by powers of $h$ computations are carried out on the lattice for $(\wedge,\delta)$ in \S{4} and for $(\wedge,\d)$ in \S{5}.  Chain and cochain
 mappings are  given in \S7 between the algebras of chains on lattices at two different scales
and  extended to give \textbf{ binary QFT algebra} morphisms, i.e. with the estimates, between the higher bracket structures at various scales.

\textbf{Acknowledgements }
 The structure of the cellular combinatorics from  \cite{S} and computations of its cumulant structure  here, yield the $h$-divisibility properties. These in turn require the special nature (\S2.2, \S2.3) of the cumulant bijection,  discussed in   \cite{RS} and in \cite{R2}. The functorial aspects of the cumulant bijection are used first in \cite{S14} and then in \S7 to obtain a direct and an inverse system  of binary QFT algebras in each dimension $n$  and with morphisms between scales.

 The cumulant bijection was directly inspired by a CUNY Einstein Chair online lecture of  Jae Suk Park in 2012 which    eventually become his formulation of \textbf{binary QFT algebras} in 2018 (see \cite{P} pp 19,20).  The  results here  are for specific combinatorial constructions from which the continuum algebra emerges.
 Some aspects fit with Markl \cite{M} in the context of homotopical algebra. The constructions here  are  canonical,   explicit and rigid. We would like to thank the referee for their comments on making the exposition more focused and succinct.

\section{General constructions}
Suppose that $V$ is a graded commutative associative algebra $(V,m)$ with square-zero map $\d:V\longrightarrow{}V$ of degree $\pm1$.  The $\Z$-grading in $V$ is denoted by $|\cdot|$. Write $\pm^{x,y}$ for the sign $(-1)^{|x|.|y|}$. (Graded) commutativity means as usual that $yx=\pm^{x,y}xy$ where we write $xy$ for $m(x\otimes{}y)$.

In this section we will discuss the infinitesimal cumulants of $\partial$ (namely its higher deviations from being a derivation) and in particular see how they can be viewed as ``Taylor coefficients of a coderivation" of square zero on  $S^*V$, with   ``$\ldots$" being explained below. To keep the paper self-contained, to establish our notation and to give the formulae for the first few terms we  have  written subsections 2.1-2.3 with their proofs. Much of this is familiar to various experts. See \cite{BL} especially  and also \cite{M} for  interesting general perspectives.  The specific  estimates here  on statistical cumulants   using these subsections    plus \S\S3,4,5,7 were the goal and  seem to be new.

\subsection{Comultiplication on $S^*V$}
Consider the tensor algebra $T^*V = \bigoplus_{k=1}^{\infty}V^{\otimes k}$.  The grading on $V$ induces a grading on $T^*V$, so that $x_1\otimes\ldots\otimes{}x_k$ has grading $|x_1|+\cdots+|x_k|$. There is a (graded) coassociative comultiplication $\Delta$ on $T^*V$ of degree zero defined by
\[\Delta(x_1\otimes\cdots\otimes{}x_k)
=\sum_{i=1}^{n-1}(x_1\otimes\cdots\otimes{}x_i)\otimes(x_{i+1}\otimes\cdots\otimes{}x_k)
\]
 Consider $S^kV\subset{}V^{\otimes{}k}$ as the subspace spanned by
\[
x_1\wedge\ldots\wedge{}x_k\equiv{}\sum\limits_{\sigma\in{}S_k}
\pm^{(\sigma,x)}x_{\sigma(1)}\otimes\cdots\otimes{}x_{\sigma(k)}
\]
for $x_1,\ldots,x_k\in{}V$, where $\pm^{(\sigma,x)}\equiv\prod_{i>j,\sigma^{-1}(i)<\sigma^{-1}(j)}
(-1)^{|x_i|\cdot|x_j|}$. The grading on $T^*V$ induces a grading on $S^*V$, so that $x_1\wedge\ldots\wedge{}x_k$ has grading $|x_1|+\cdots+|x_k|$.
The coassociative comultiplication $\Delta$ defined above on $T^*V$ induces a cocommutative coassociative comultiplication of degree $0$ on $S^*V$.
Explicitly,
\[
\Delta(v_1\wedge\ldots\wedge{}v_k)=\sum_I\pm^Iv_I\otimes{}v_{I^c}
\]
the sum being over all $2^k-2$  proper non-empty subsets $I$ of $[k]\equiv\{1,\ldots,k\}$, whose complement is denoted $I^c$, while if the elements of $I$ are written in increasing order as $I=\{i_1,\ldots,i_r\}$ then $v_I\equiv{}v_{i_1}\wedge\ldots\wedge{}v_{i_r}$. The sign is
\[
\pm^I=\prod\limits_{i\in{}I,j\in{}I^c,i>j}(-1)^{|v_i|\cdot|v_j|}\eqno{(1)}
\]
Verification of cocommutativity, $P\circ\Delta=\Delta$ where $P(v\otimes{}w)=(-1)^{|v|\cdot|w|}w\otimes{}v$ follows from the identity  $\pm^I\pm^{I^c}=(-1)^{|v_I|\cdot|v_{I^c}|}$.

\medskip\begin{example} The comultiplication applied to an element of $S^3V$ is a sum of $2^3-2=6$ terms. Thus $\Delta(v_1\wedge{}v_2\wedge{}v_3)$ is
\begin{align*}
&v_1\otimes{(}v_2\wedge{}v_3)
+(-1)^{|v_1|\cdot|v_2|}v_2\otimes{}(v_1\wedge{}v_3)
+(-1)^{(|v_1|+|v_2|)|v_3|}v_3\otimes{}(v_1\wedge{}v_2)\\
&+(v_1\wedge{}v_2)\otimes{}v_3
 +(-1)^{|v_2|\cdot|v_3|}(v_1\wedge{}v_3)\otimes{}v_2
+(-1)^{|v_1|(|v_2|+|v_3|)}(v_2\wedge{}v_3)\otimes{}v_1
\end{align*}
\end{example}

\subsection{Cumulant bijection}
Let $\tau_1:S^*V\longrightarrow{}V$ be the multiplication, that is, $\tau_1(v_1\wedge\cdots\wedge{}v_n)=v_1\cdots{}v_n$.
Following \cite{RS}, define the {\sl cumulant bijection} $\tau:S^*V\longrightarrow{}S^*V$ to be the coalgebra lift of $\tau_1$ to $S^*V$.
This means that  $\tau\equiv\sum\limits_{r=1}^\infty\tau_r$ where $\tau_r:S^*V\longrightarrow{}S^rV$ is defined for $r>1$ by
\[
    \tau_r=(\tau_1)^{\otimes{}r}\circ\Delta^{r-1}
\]
where $\Delta^{r-1}:S^*V\longrightarrow{}(S^*V)^{\otimes{}r}$ is the iterated coproduct defined by
$\Delta^1=\Delta$ and $\Delta^{n+1}=(\Delta\otimes\id^{\otimes{}n})\circ\Delta^n$.
That $\tau:S^*V\longrightarrow{}S^*V$ is a coalgebra map means that $(\tau\otimes\tau)\circ\Delta=\Delta\circ\tau$.

\begin{example} On monomials of order up to three, the formulae for $\tau$ are
\begin{align*}
    \tau(a)&=a\\
    \tau(a\wedge{}b)&=ab+a\wedge{}b\\
    \tau(a\wedge{}b\wedge{}c)&=abc+a\wedge{}bc+(-1)^{|a|\cdot|b|}b\wedge{}ac+(-1)^{(|a|+|b|)|c|}c\wedge{}ab+a\wedge{}b\wedge{}c
\end{align*}
As a linear map $\tau:S^*V\longrightarrow{}S^*V$, its action is block upper-triangular where the $r$-th  block is $S^rV$, and the diagonal blocks are the identity. Thus $\tau$ is invertible and its inverse looks like
\begin{align*}
    \tau^{-1}(a)&=a\\
    \tau^{-1}(a\wedge{}b)&=a\wedge{}b-ab\\
    \tau^{-1}(a\wedge{}b\wedge{}c)&=a\wedge{}b\wedge{}c\!-\!a\wedge{}bc
    \!-\!(-1)^{|a|\cdot|b|}b\wedge{}ac\!-\!(-1)^{(|a|+|b|)|c|}c\wedge{}ab+2abc
\end{align*}
\end{example}

In particular, the coefficients in the formulae for $\tau$ are always $\pm1$ while those in the formulae for $\tau^{-1}$ are more complicated, depending on the relevant partition. The component of $\tau^{-1}(v_1\wedge\cdots\wedge{}v_k)$ in $V$ (with no wedges) is a multiple (depending on $k$) of the product $\tau_1(v_1\wedge\cdots\wedge{}v_k)=v_1\cdots{}v_k$.

\medskip\begin{lemma} On $S^kV$, the actions of $\tau$ and $\tau^{-1}$ are given explicitly by
\begin{align*}
\tau(v_1\wedge\cdots\wedge{}v_k)&=\sum\limits_{r=1}^k
\sum\limits_{I_1\cup\cdots\cup{}I_r=[k]}
\pm\tau_1(v_{I_1})\wedge\cdots\wedge\tau_1(v_{I_r})\\
\tau^{-1}(v_1\wedge\cdots\wedge{}v_k)&=\sum\limits_{r=1}^k
\sum\limits_{I_1\cup\cdots\cup{}I_r=[k]}
\left(\prod\limits_{p=1}^rd_{|I_p|}\right)
\pm\tau_1(v_{I_1})\wedge\cdots\wedge\tau_1(v_{I_r})
\end{align*}
where $d_k\equiv(-1)^{k-1}(k-1)!$, the summations are over all (unordered) partitions of $[k]$ into $r$ non-empty subsets, and the sign is that acquired in changing the order of the $v's$ from sequential in $v_1\wedge\cdots\wedge{}v_k$ to that in $v_{I_1}\wedge\cdots\wedge{}v_{I_r}$.
In particular, $p_1\circ\tau^{-1}=(-1)^{k-1}(k-1)!\tau_1$, where $p_1$ denotes the projection map $S^*V\longrightarrow{}V$ on the first component.\end{lemma}
\begin{proof}
The formula for $\tau$ follows immediately from its definition above,
$\tau=\sum_{r=1}^\infty\big((\tau_1)^{\otimes{}r}\circ\Delta^{r-1}\big)$ when combined with the action of the coproduct in \S2.1. To prove the formula for $\tau^{-1}$, it suffices to verify that its composition with $\tau$ is the identity. Let $X$ denote the expression given on the right hand side of the second equation in the statement of the lemma; it suffices to prove that $\tau(X)=v_1\wedge\cdots\wedge{}v_k$. Computing, using the formula for $\tau$,
\[\tau(X)
=\sum\limits_{r=1}^k
\sum\limits_{I_1\cup\cdots\cup{}I_r=[k]}
\sum\limits_{s=1}^r
\sum\limits_{J_1\cup\cdots\cup{}J_s=[r]}
\frac1{r!s!}\left(\prod\limits_{p=1}^rd_{|I_p|}\right)
\pm\tau_1(v_{K_1})\wedge\cdots\wedge\tau_1(v_{K_s})\]
where $K_a\equiv\bigcup_{p\in{}J_a}I_p$.
Here and throughout the rest of this proof, the summations are over ordered partitions and hence extra factors have been introduced (reciprocals of factorials). Interchanging the summations over $s$ and $r$ and extracting the summation over $K_1,\ldots,K_s$ leaves
\begin{align*}
\sum\limits_{s=1}^k\sum\limits_{K_1\cup\cdots\cup{}K_s=[k]}
&\pm\frac1{s!}\tau_1(v_{K_1})\wedge\cdots\wedge\tau_1(v_{K_s})
\cdot\sum\limits_{r=s}^k\frac1{r!}
\sum\limits_{{I_1\cup\cdots\cup{}I_r=[k]\atop{J_1\cup\cdots\cup{}J_s=[r]}}
\atop{}K_a=\bigcup_{p\in{}J_a}I_p}
\left(\prod\limits_{p=1}^rd_{|I_p|}\right)
\end{align*}
the inner summation being over partitions $\{I_p\}$ and $\{J_a\}$ which match given $\{K_j\}$.  It suffices to show that the coefficient on the right hand side vanishes unless $s=k$ and the $K_a$ are singletons when it is 1. For a fixed partition $K_1\cup\cdots\cup{}K_s=[k]$ and any (ordered) partition $\{J_a\}$ of $[r]$, a matching partition $\{I_p\}$ is given by splitting each $K_a$ ($a=1,\ldots,s$) into $|J_a|$ subsets. The coefficient of the term coming from the particular partition $K_1\cup\cdots\cup{}K_s$ is thus
\[
\sum\limits_{r=s}^k\frac1{r!}
\sum\limits_{J_1\cup\cdots\cup{}J_s=[r]}\prod\limits_{a=1}^s\Bigg(
\sum\limits_{{ordered\ partitions\ of}\atop{K_a\ into\ |J_a|\ sets}}\quad\prod\limits_{p=1}^{|J_a|}d_{|\hbox{$p^{th}$ set}|}\Bigg)
\]
The inner sum depends on $\{J_a\}$ only by the orders of the sets, giving
\begin{align*}
&\sum\limits_{r=s}^k\frac1{r!}
\sum\limits_{j_1+\cdots+j_s=r}\frac{r!}{\prod\limits_{a=1}^sj_a!}
\prod\limits_{a=1}^s\Bigg(
\sum\limits_{{ordered\ partitions\ of}\atop{K_a\ into\ j_a\ sets}}\quad\prod\limits_{p=1}^{j_a}d_{|\hbox{$p^{th}$ set}|}\Bigg)
\end{align*}
Interchanging the first two summations leaves
\begin{align*}
&\sum\limits_{j_1,\ldots,j_s\geq1}
\prod\limits_{a=1}^s\Bigg(\frac1{j_a!}
\sum\limits_{{ordered\ partitions\ of}\atop{K_a\ into\ j_a\ sets}}\quad\prod\limits_{p=1}^{j_a}d_{|\hbox{$p^{th}$ set}|}\Bigg)
\end{align*}
which is a product over $a=1,\ldots,s$ of
\[\sum\limits_{j\geq1}\frac1{j!}
\sum\limits_{{ordered\ partitions\ of}\atop{K_a\ into\ j\ sets}}\quad\prod\limits_{p=1}^{j}d_{|\hbox{$p^{th}$ set}|}
=\sum\limits_{j\geq1}\frac1{j!}
\sum\limits_{b_1+\cdots+b_j=|K_a|}\frac{|K_a|!}{\prod_{p=1}^jb_p!}\prod\limits_{p=1}^{j}d_{b_p}\>.\]
Taking out the constant $|K_a|!$ from the summations and using the formula $d_b=(-1)^{b-1}(b-1)!$,
the inner sum is seen to become the coefficient of $x^{|K_a|}$ in $\left(\sum_{b=1}^\infty(-1)^{b-1}\tfrac1bx^b\right)^j=\big(\ln(1+x)\big)^j$ and so the whole expression sums to $|K_a|!$ times the coefficient of $x^{|K_a|}$ in $\sum_{j\geq1}\tfrac1{j!}\big(\ln(1+x)\big)^j=x$ which therefore vanishes unless $|K_a|=1$ and in that case is $1$.
\end{proof}

\subsection{Higher infinitesimal cumulants}
Extend the differential $\d$ on $V$ to a coderivation on $S^*V$ with respect to $\Delta$. Explicitly, define  $\d^\wedge:S^*V\longrightarrow{}S^*V$ by,
\[
 \d^\wedge(v_1\wedge\ldots\wedge{}v_k)\equiv\sum\limits_{i=1}^k
(-1)^{\sum_{j<i}|v_j|}v_1\wedge\ldots\wedge{}\d{}v_i\wedge\ldots\wedge{}v_k
\]
Since $\d$ is a square-zero map of degree $\pm1$, the same follows about $\d^\wedge$. Further,
$\d^\wedge$ is a coderivation on  $S^*V$ with respect to  $\Delta$, meaning that  \[\Delta\circ{}\d^\wedge=(\d^\wedge\otimes\id+\idb\otimes{}\d^\wedge)\circ\Delta
\]
where $\idb:S^*V\longrightarrow{}S^*V$, $\idb(v)=(-1)^{|v|}v$. Since $\tau$ preserves the grading, it commutes with $\idb$. Note that  the extra signs (from $\idb$) in the above equality appear due to the natural Koszul sign and can be subsumed by using an appropriate definition of composition; however we want to make statements as transparent as possible, so we use the symbol $\circ$ to mean the usual functional composition only.

\noindent The conjugation of a coderivation by a coalgebra  isomorphism  is a coderivation, so that $D\equiv{}\tau^{-1}\d^\wedge\tau$ is a square-zero coderivation of $S^*V$. Furthermore, a coderivation $D$ on $S^*V$ is uniquely determined by its projection on the first factor $p_1\circ{}D:S^*V\longrightarrow{}V$. Thus it is determined by its homogeneous components $\d_k\equiv(p_1\circ{}D)|_{S^kV}:S^kV\longrightarrow{}V$, known as the \textbf{Taylor components or coefficients } of $D$.

\begin{definition}
\cite{RS} Write $D\equiv\tau^{-1}\d^\wedge\tau$ and denote its $k^{th}$ Taylor component $S^kV\longrightarrow{}V$ by $\d_k$. We will call $\d_k$ the $k$-bracket induced by the pair of graded commutative associative algebra $(V,m)$ and square-zero map $\d$.
\end{definition}

\noindent The map $\d_{k+1}$ will  be referred to as the $k^{th}$ {\sl infinitesimal cumulant} of $\d$ with respect to the multiplication $m$; this  is  reasonable because  they measure infinitesimally the deviation of the exponential from being a (bijective) product-preserving mapping, as can be seen in the formulae for $\d_k$ given in the following example.

\begin{example} On monomials, $D$ acts by
\begin{align*}
    D(a)&=\d{}a\\
    D(a\wedge{}b)&=\d(ab)-\d{}a\cdot{}b+\d{}a\wedge{}b+(-1)^{|a|\cdot|b|}\big(\d{}b\wedge{}a-\d{}b\cdot{}a\big)
\end{align*}
while $D(a\wedge{}b\wedge{}c)$ is a sum of 19 terms. Its Taylor coefficients are
\begin{align*}
    \d_1(a)&=\d{}a\\
    \d_2(a\wedge{}b)&=\d(ab)-\d{}a\cdot{}b-(-1)^{|a|\cdot|b|}\d{}b\cdot{}a\\
    \d_3(a\wedge{}b\wedge{}c)&=\d(abc)-\d(ab)\cdot{}c-(-1)^{|b|\cdot|c|}\d(ac)\cdot{}b
    -(-1)^{|a|(|b|+|c|)}\d(bc)\cdot{}a\\
    &\quad+(-1)^{(|a|+|b|)|c|}\d{}c\cdot{}ab+(-1)^{|a|\cdot{}|b|}\d{}b\cdot{}ac+\d{}a\cdot{}bc
\end{align*}
For example, $\d_2$ is the first infinitesimal cumulant and expresses the (first) deviation from Leibniz, equivalently $\d_2(a\wedge{}b)=\d(ab)-\d{}a\cdot{}b-(-1)^{|a|}a\cdot{}\d{}b$.
\end{example}

\begin{lemma}\cite{RS} The map $D$ is a square-zero coderivation of $S^*V$ and its $k^{th}$  Taylor coefficient is given by the following formula where $\pm^I$ is as defined in (1)
\[\d_k(v_1\wedge\cdots\wedge{}v_k)
=\sum_{r=1}^k(-1)^{k-r}
\sum\limits_{|I|=r}\pm^I\ {}\d(\tau_1(v_I))\cdot\tau_1(v_{I^c})\eqno{(2)}\]
Here the inner summation is over order $r$ subsets $I\subseteq[k]$.
\end{lemma}
\begin{proof}
By Definition 2.4, $D=\tau^{-1}\circ{}\d^\wedge\circ\tau$ from which it follows immediately that $D$ has square zero and is a coderivation. Its $k^{th}$ Taylor coefficient is $\d_k=p_1\circ{}D=p_1\circ\tau^{-1}\circ{}\d^\wedge\circ\tau$ on $S^kV$,
which we compute from Lemma 2.3 for $\tau$,
\begin{align*}
\d_k(v_1\wedge\cdots\wedge{}v_k)
&=(p_1\circ\tau^{-1}\circ\d^\wedge)\sum\limits_{r=1}^k
\sum\limits_{I_1\cup\cdots\cup{}I_r=[k]}
\pm\tau_1(v_{I_1})\wedge\cdots\wedge\tau_1(v_{I_r})
\end{align*}
where the sum is over all (unordered) partitions of $[k]$ into $r$ non-empty subsets, and the sign is that acquired in changing the order of the $v$'s from sequential in $v_1\wedge\cdots\wedge{}v_k$ to that in $v_{I_1}\wedge\cdots\wedge{}v_{I_r}$. Equivalently, introducing a factor $\tfrac1{r!}$, one may sum over ordered partitions of $[k]$ into $r$ non-empty subsets, which we now do. By the definition of $\d^\wedge$ above, this expands as
\[
(p_1\circ\tau^{-1})\sum\limits_{r=1}^k
\sum\limits_{I_1\cup\cdots\cup{}I_r=[k]}\sum\limits_{p=1}^r
\pm\frac1{r!}\d\big(\tau_1(v_{I_p})\big)\wedge\tau_1(v_{I_1})\wedge\cdots
\widehat{\tau_1(I_p)}\wedge\cdots\wedge\tau_1(v_{I_r})
\]
where the hat indicates a missing term, and the sign is that acquired by the (new) change in order of the $v$'s, where $v_{I_p}$ is now listed first. Applying Lemma~2.3 for $\tau^{-1}$, this becomes
\begin{align*}
&\sum\limits_{r=1}^k
\sum\limits_{I_1\cup\cdots\cup{}I_r=[k]}\sum\limits_{p=1}^r
\pm(-1)^{r-1}\frac1r\cdot\d\big(\tau_1(v_{I_p})\big)\tau_1(v_{I_p^c})\\
&=\sum\limits_{I\subseteq[k]\atop{}I\not=\emptyset}\sum\limits_{r=1}^k\sum\limits_{p=1}^r
\sum\limits_{{I_1\cup\cdots\cup{}I_r=[k]}\atop{}I_p=I}
\pm^I(-1)^{r-1}\frac1r\cdot\d\big(\tau_1(v_{I})\big)\tau_1(v_{I^c})\\
&=\sum\limits_{I\subseteq[k]\atop{}I\not=\emptyset}\sum\limits_{r=1}^{k+1-|I|}
(-1)^{r-1}N_{k-|I|,r-1}\cdot\pm^I\d\big(\tau_1(v_{I})\big)\tau_1(v_{I^c})
\end{align*}
where $N_{s,t}$ denotes the number of ordered partitions of a set of order $s$ into $t$ non-empty sets, or equivalently the number of onto maps $[s]\longrightarrow[t]$.

\noindent To complete the proof of the lemma, we need to show that the inner sum is $(-1)^{k-|I|}$, that is that
$\sum_{r=1}^{s+1}(-1)^{r-1}N_{s,r-1}=(-1)^s$ for $s\geq0$.
For $s=0$, this is $N_{0,0}=1$ while for $s>0$, it is the identity
\[N_{s,s}-N_{s,s-1}+\cdots\pm{}N_{s,1}=1\eqno{(3)}\]
For completeness, we include a proof. By the inclusion-exclusion principle, we can count onto maps $[s]\longrightarrow[t]$ by counting all maps ($t^s$) and excluding those in the union of $t$ sets, namely those whose image does not include $1,2,\ldots,t$, respectively. This gives
\[N_{s,t}=t^s-t\cdot(t-1)^s+\binomial{t}{2}(t-2)^s-\cdots
=\sum\limits_{k=1}^t(-1)^{t-k}\binomial{t}{k}k^s\]
Taking an alternating sum and interchanging the summations now gives
\[
\sum\limits_{t=1}^s(-1)^{s-t}N_{s,t}
=\sum\limits_{k=1}^s\sum\limits_{t=k}^s(-1)^{s-k}\binomial{t}{k}k^s
=\sum\limits_{k=1}^s(-1)^{s-k}\binomial{s+1}{k+1}k^s\]
which is the coefficient of $x^s$ in $(1-x)^{s+1}\sum\limits_{k=1}^\infty{}k^sx^k
=(1-x)^{s+1}\cdot\Big(x\tfrac{d}{dx}\Big)^s\left(\tfrac1{1-x}\right)$.
Denote this expression by $p_s(x)$ so that
$\Big(x\tfrac{d}{dx}\Big)^s\left(\tfrac1{1-x}\right)
=\tfrac{p_s(x)}{(1-x)^{s+1}}$ and observe that
$x\frac{d}{dx}\left(\frac{p_{s-1}(x)}{(1-x)^s}\right)
=\frac{sx\cdot{}p_{s-1}(x)+x(1-x)\cdot{}p_{s-1}'(x)}{(1-x)^{s+1}}$
so that $p_0(x)=1$ while $p_s=sx\cdot{}p_{s-1}+x(1-x)p_{s-1}'$. We see inductively that $p_s$ is a polynomial of degree $s$. The recurrence relation on $p_s$ induces one on its leading coefficient $a_s$, namely $a_s=sa_{s-1}-(s-1)a_{s-1}=a_{s-1}$ while $a_0=1$. Thus $a_s=1$ for all $s$.  This completes the proof of (3).
\end{proof}

\begin{lemma} Fixing the last $(k-2)$ components, $\d_k$ as a map $V\otimes{}V\longrightarrow{}V$ is the commutator of $\d_{k-1}$ (as a map $V\longrightarrow{}V$) with multiplication; that is,
\begin{align*}
\d_k(v_1\wedge{}v_2\wedge{}v_3\wedge\ldots\wedge{}v_k)
=&\d_{k-1}(v_1v_2\wedge{}v_3\wedge\ldots\wedge{}v_k)\\
&-(-1)^{|v_1|}v_1\cdot{}\d_{k-1}(v_2\wedge{}v_3\wedge\ldots\wedge{}v_k)\\
&-(-1)^{|v_2|(1+|v_1|)}v_2\cdot{}\d_{k-1}(v_1\wedge{}v_3\wedge\ldots\wedge{}v_k)
\end{align*}
\end{lemma}
\begin{proof}
By Lemma 2.6, the left hand side can be written as a sum of terms indexed by non-empty subsets $I\subseteq[k]$. This sum can be split according as 1 and/or 2 are elements of $I$,
\[\d_k(v_1\wedge\cdots\wedge{}v_k)
=\Big(\sum\limits_{{I\subseteq[k],I\not=\emptyset}\atop{1,2\not\in{}I}}
+\sum\limits_{{I\subseteq[k]}\atop{1,2\in{}I}}
+\sum\limits_{{I\subseteq[k]}\atop{1\in{}I,2\not\in{}I}}
+\sum\limits_{I\subseteq[k]\atop{2\in{}I,1\not\in{}I}}\Big)
\pm^I(-1)^{k-|I|}\d(v_I)\cdot{}v_{I^c}
\]
Here we omitted the symbols $\tau_1$ denoting repeated multiplication, which is being written as concatenation. Denote these four sums $\Sigma_1$, $\Sigma_2$, $\Sigma_3$, $\Sigma_4$ respectively. Similarly, each term on the right hand side of the equality to be proved can be expanded using lemma 2.6 as a sum of $2^{k-1}-1$ terms,  over non-empty subsets of a $(k-1)$-element set, and further each of these sums can be split into two according as the subset does not or does contain the first element. This gives the three terms on the RHS of the equality to be proved as, $-\Sigma_1+\Sigma_2$, $\Sigma_1+\Sigma_4$ and $\Sigma_1+\Sigma_3$, respectively.
\end{proof}
This lemma will be useful in proving the inductive steps when we come to compute $\d_k$ on the lattice in \S{4} and \S{5}.

\begin{definition}
When the $k^{th}$ infinitesimal cumulant of $\d$ vanishes (that is, $\d_{k+1}=0$) we say that $\d$ has Grothendieck order at most $k$ with respect to $m$.
\end{definition}
Extending $\d_k:S^kV\longrightarrow{}V$ to a coderivation on $S^*V$ yields a map $\d_k^\wedge:S^*V\longrightarrow{}S^*V$ which more precisely acts as $\d_k^\wedge:S^rV\longrightarrow{}S^{r-k+1}V$ and has $D=\sum_{k=1}^\infty{}\d_k^\wedge$. The general formula for the action of $\d_n^\wedge$ is (where $\pm^I$ is as defined in (1))
\[
\d_n^\wedge(v_1\wedge\cdots\wedge{}v_k)
=\sum\limits_{I\subset[k]\atop{}|I|=n}\pm^I\d_n(v_I)\wedge{}v_{I^c}\eqno{(4)}
\]

The map $\d_k^\wedge:S^*V\longrightarrow{}S^*V$ has Grothendieck order at most $k$ with respect to the multiplication $\wedge$ on $S^*V$.

\begin{example}
 $\d_1^\wedge=\d^\wedge$, while $\d_2^\wedge$ acts on monomials up to third order by
\begin{align*}
    \d_2^\wedge(a)&=0\\
    \d_2^\wedge(a\wedge{}b)
    &=\d(ab)-\d{}a\cdot{}b-(-1)^{|a|\cdot|b|}\d{}b\cdot{}a\\
    \d_2^\wedge(a\wedge{}b\wedge{}c)
    &=\d(ab)\wedge{}c+(-1)^{|a|}a\wedge{}\d(bc)
    +(-1)^{|b|\cdot|c|}\d(ac)\wedge{}b\\
    &\quad-\d{}a\cdot{}b\wedge{}c
    -(-1)^{|b|\cdot|c|}\d{}a\cdot{}c\wedge{}b
    -(-1)^{|a|\cdot|b|}\d{}b\cdot{}a\wedge{}c\\
    &\quad-\!(\!-\!1)^{|a|}a\!\wedge\!{}\d{}b\cdot{}c
    \!-\!(\!-\!1)^{(|a|+|b|)|c|}\d{}c\cdot{}a\!\wedge\!{}b
    \!-\!(\!-\!1)^{(|a|+|b|)|c|}a\!\wedge\!{}\d{}c\cdot{}b\\
    &=\d_2(a\wedge{}b)\wedge{}c+(-1)^{|a|}a\wedge{}\d_2(b\wedge{}c)
    +(-1)^{|b|\cdot|c|}\d_2(a\wedge{}c)\wedge{}b
\end{align*}
\end{example}
 In conclusion, we have the following general structure (see the beautiful paper \cite{Koszul}) to be estimated in  \S\S4,5,7 for the examples of \S3.
\begin{theorem} For any graded commutative associative algebra $(V,m)$ with square-zero map $\d$ of grading $\pm1$, $S^*V$ is a cocommutative coassociative coalgebra with maps $\d_k^\wedge:S^rV\longrightarrow{}S^{r-k+1}V$ as defined above, for which  $D=\sum_{k=1}^\infty{}\d_k^\wedge$ is a coderivation of square zero while each $\d_k^\wedge$ has Grothendieck order at most $k$ with respect to $\wedge$ and $\d_1^\wedge=\d^\wedge$. The Taylor coefficients of $D$ are maps $\d_k:S^kV\longrightarrow{}V$ which define $k^{th}$ order multi-brackets and describe the $(k-1)^{th}$  infinitesimal cumulant of $\partial$ (higher order deviation from Leibniz).
\end{theorem}

\section{The geometry of the lattice chain algebra  and  the  lattice cochain algebra}
\textbf{Discretization of the continuum algebra in dimension three.}
    In  \cite{S}, one builds the vector calculus of the  lattice model defined by all of the $2h$ cubes in the $h$ cubical lattice of $\R^3$, from all of these larger cubes, their faces, their edges and their vertices  by forming  vector spaces $C_k$  which are ``up to sign" generated by these so-called $k$-cubes, for
 $k = 3,2,1$ and $0$.  For  $k = 0$, $C_k$  is actually  generated by  all of the $h$-lattice points  and for   $k= 1,2,3$,  $C_k$  is  generated by \textbf{oriented} $k$-cubes of edge length $2h$ modulo the relations
 ($k$-cube, orientation) = --($k$-cube, opposite orientation). These are  the {\sl chains}.

 One obtains finite dimensionality by boundary conditions, here  by identifying periodically, with period which is the same  power of two in each direction.

 The dual spaces $C^k$, consisting of linear functionals on these finite dimensional  $C_k$,   are called  {\sl cochains}. There is also a star operator $*$ on each, changing $k$ to $3-k$, and if one identifies chains and cochains using the basis, then star conjugates $\d$ to $\delta$ and {\it vice versa} since  $*^2=1$. See the upper row of the  figure.
\[
\includegraphics[width=.12\textwidth]{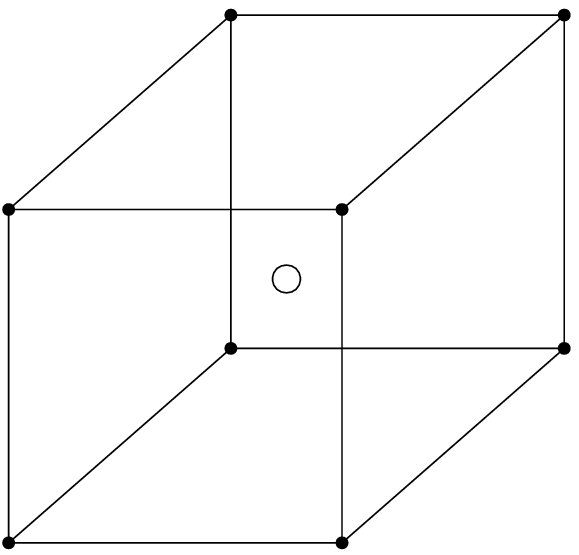}\hskip4em\includegraphics[width=.15\textwidth]{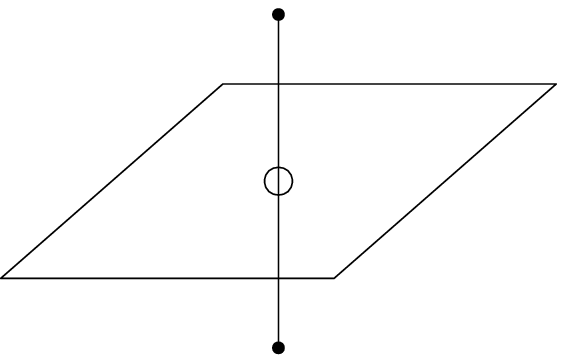}
\]
 \[\includegraphics[width=.13\textwidth]{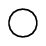}\hskip2em\includegraphics[width=.13\textwidth]{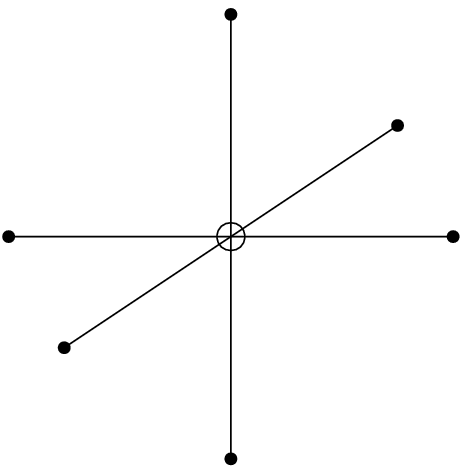}\hskip2em\includegraphics[width=.16\textwidth]{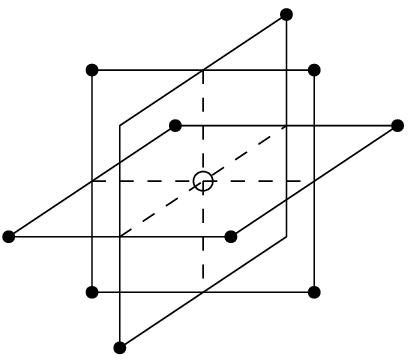}\hskip2em
\includegraphics[width=.13\textwidth]{200114-cube-2.eps}\]

 \textbf{The main  point of the choice of  combinatorics.}
 \noindent Firstly, there  is for  cochains,  defined on
 $k$-cubes of side length $2h$,  a  graded commutative and associative product which makes  use of the linear algebra of  alternating tensors. This is so because at each  lattice point there are exactly (1, 3, 3, 1)   $k$-cubes  for $k=0,\ 1,\ 2,\ 3$,  of edge length $2h$ with centre point at that grid point. See the second row of the figure. This gives a direct sum decomposition of the $2h$-cochains on the $h$-lattice  into the exterior algebras of the cotangent spaces of $\R^3$ at the lattice points. There is then the obvious direct sum  exterior algebra structure on cochains, which is graded commutative and associative.

  The operator $\delta$ is not  a derivation of the product, which  relation  only appears  on taking the calculus limit. The error is order $h$, because there are  $h$-shifts in the  true Leibniz type  formula for $\delta$ of a product.

Secondly, the  $h$-lattice chains generated by the bigger $k$-cubes of edge size $2h$    have a direct sum decomposition into the exterior algebras on the tangent spaces at the lattice points. Thus the chains  also have a  graded commutative and associative algebra structure. This product  is not  familiar in  combinatorial or algebraic  topology but its continuum analogue appears in differential geometry.

    One considers for chains,  as one  does for  continuum  multivector fields,  a second product or bracket  [ , ], defined  as   the deviation of $\d$ on chains from being a derivation of the   commutative and associative product on chains. Tautologically,  $\d$ is a derivation of this  bracket  because $\d^2=0$. Note this means the  bracket  [~,~] of two cycles is not only a cycle but is  canonically a boundary, being $\d$ of the exterior product of the two cycles. (See Appendix for the continuum discussion.)

{\bf Discretisation in $n$-dimensions.} Consider an $n$-dimensional periodic cubic lattice with grid step $h$ and a size divisible by four in each direction. That is, the vertices are  $L_h=(h\Z)^n/(4Nh\Z)^n$ which is naturally bicoloured by $\frac1h\sum_{i=1}^nx_i$ modulo 2 at $(x_i)\in{}L_h$. This bicolouring splits the lattice as a union of two interlocking  lattices, here pictured for $n=3$.
\[\includegraphics[width=.25\textwidth]{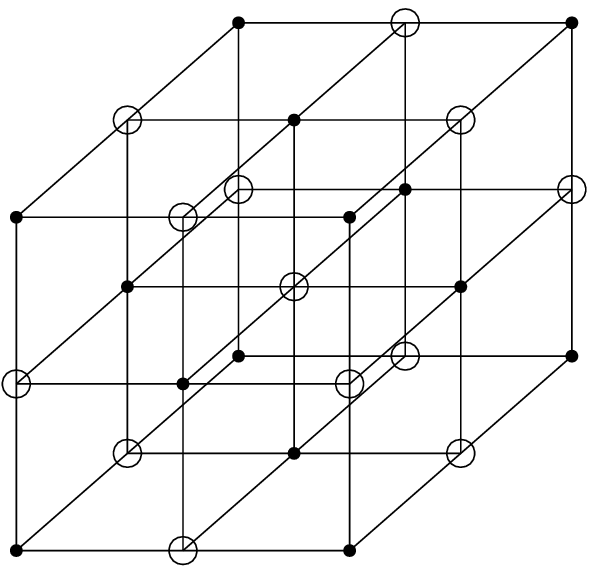}
\]
Consider the subcomplex of the chain complex of all cellular chains, in which $r$-chains are generated by $r$-dimensional cubes of edge length $2h$ (whose vertices all have the same colour). Such an  $r$-cube is specified by its centre, a point in $L_h$, and its type, the subset of $[n]$ consisting of those directions for which the projection of the cube on the corresponding axis is an interval of length $2h$ (rather than a point).  Since these cubes are Cartesian products of intervals, choosing an ordering on $[n]$ will induce natural orientations on the cubes.

Let $C$ denote the graded vector space of chains, graded by dimension, with basis $\{I_a\}$, where $I_a$ labels the cell of type $I\subset[n]$ whose centre is at lattice point $a$.
Denote by $Ia$ the translation of $a$ forward by $h$ in each of the directions in $I$ and by $I'a$, the translation of $a$ backward by $h$ in each of the directions in $I$. By abuse of notation, when $I$ is a singleton $\{u\}$, we will write $Ia$ and $I'a$ as $ua$ and $u'a$, respectively.
The geometric boundary is given on generators by
\[
\d(I_a)=\sum\limits_{u\in{}I}\pm^u_I\left((I\bs{}u)_{ua}-(I\bs{}u)_{u'a}\right)
\]
where the sign $\pm^u_I$ is such that $u\cdot(I\bs{}u)=\pm^u_I\cdot{}I$, that is $\pm^u_I=(-1)^{\#\{i\in{}I|i<u\}}$. The homology of this chain complex has dimension $2^n$ in degree zero because of the way we have constructed cells of edge length $2h$.

\medskip\noindent{\bf The chain algebra.} The vector space of chains $C$ naturally splits as a direct sum $C=\oplus_{a\in{}L_h}C_a$ where $C_a=\wedge^*\R^n$ consists of those chains which are combinations of cubical cells centred at $a$. The (pointwise) wedge product makes $C$ into a graded commutative associative algebra.

A generic chain can be written as \[\sum\limits_{I,a}f_I(a)\cdot{}I_a\] where we combine the coefficients $f_I(a)$ of the basis chains $I_a$ into scalar-valued functions $f_I$ on lattice points, defined for all $2^n$ subsets $I$ of $[n]$.

\noindent{\bf The cochain algebra.} Let $A$ denote the graded vector space of
scalar-valued cochains on the lattice (vector space dual to $C$). An element of $A$ is (determined by) a function $f$
which evaluates on chains $I_a$ to a scalar $f(I_a)$, or equivalently $2^n$ functions $f_I$, each a
scalar-valued function on the lattice (here $f_I(a)\equiv{}f(I_a)$). One could write the
corresponding cochain as
\[
\sum\limits_{I}f_IdI
\]
using the notation of differential forms. The graded vector space $A$ splits naturally as
$A=\oplus_{a\in{}L_h}A_a$ where each $A_a=\wedge^*\R^n$ consists of the cochains supported on
cells centred at $a$. Use the pointwise wedge product to define a multiplication, $m$, on $A$;
on zero-cochains it gives pointwise multiplication of scalar valued functions on the lattice.
Thus if $f,g\in{}A$, their product $f.g\equiv{}m(f,g)$ is defined by
\[
(f.g)(I_a)=\sum\limits_{J\cup{}K=I}\pm_{J,K}f(J_a)g(K_a)\eqno{(5)}
\]
where the sum is over all ordered pairs of disjoint sets $J$, $K$ for which $J\cup{}K=I$.
The sign $\pm_{J,K}$ is defined so that $J\cdot{}K =\pm_{J,K}(J\cup{}K)$, that is,
$\pm_{J,K}=(-1)^{\#\{(j,k)|j\in{}J,k\in{}K,j>k\}}$.
If $f$ is an $r$-cochain and $g$ is an $s$-cochain, then $f.g$ will be an $(r+s)$-cochain while
$g.f=(-1)^{rs}f.g$. So $m$ is a (graded) commutative associative product on $A$.

\noindent Since the cubic lattice is self-dual, there is a natural correspondence between the chain
complex and the cochain complex. Thus operators can be interchangeably considered as acting on the chain complex or the cochain complex. In Lemma 3.1, we choose to write all the operators as acting on cochains.

\medskip\noindent{\bf The star operator $*$.} The star operator acts on basis elements by
\[
*(dI)=\pm_{I,I^c}{}dI^c
\]
where $I^c$ denotes the complement $[n]\bs{}I$, so that $(dI)\wedge*(dI)=d[n]$. Since $\pm_{I,I^c}\cdot\pm_{I^c,I}=(-1)^{|I|.|I^c|}$, thus $*:A^k\longrightarrow{}A^{n-k}$ while $*^2=(-1)^{k(n-k)}$ on $k$-cochains. This means that for odd $n$, $*$ is an involution. On arbitrary cochains,
\[
(*f)_I=\pm_{I^c,I}f_{I^c}\eqno{(6)}
\]

\noindent{\bf The map $\d$.} Let $\Tu$ and $\Tup$ denote translations by $h\eu$ and $-h\eu$ respectively, where $\eu$ is the unit vector in $\R^n$ in the positive $u$-direction. These operators act on the lattice $L_h$ and thus also on functions on the lattice by
\[
(\Tu{}f)(p)\equiv{}f(p-h\eu),\quad{}(\Tup{}f)(p)\equiv{}f(p+h\eu)
\]
A general chain can be written as a linear combination $\sum_{I,a}f_I(a)I_a$, where $f_I$ is a scalar-valued function on the lattice $L_h$, for each $I\subseteq[n]$. The geometric boundary $\d$ acts on general chains as
\[
\d\Big(\sum\limits_{I,a}f_I(a)I_a\Big)
\!=\!\sum\limits_{I,a}\sum\limits_{u\in{}I}
               \!\pm_I^uf_I(a)\big((I\bs{}u)_{ua}\!-\!(I\bs{}u)_{u'a}\big)
\!=\!\sum\limits_{I,a}\sum\limits_{u\not\in{}I}[(\Tu\!-\!\Tup)f_{uI}](a)I_a\]
where we have extended the coefficients $f_I$ to be defined for arbitrary sequences of elements of $[n]$, rather than subsets $I$, in such a way that they vanish if the sequence contains a repeat, the sign changes under transposition of adjacent elements and is equal to the original $f_I$ when the sequence is monotonic increasing; $uJ$ is the concatenation of $u$ with $J$ in that order. Since $f$ vanishes on sequences containing a repeat, the restriction $u\not\in{}I$ can be lifted, leaving
\[
\d=\sum\limits_{u=1}^n\d_u
\quad\hbox{ where }\quad
\d_u\left(\sum\limits_{I}f_I\cdot{}I\right)=\sum\limits_I(\Tu-\Tup)f_{uI}\cdot{}I
\]
Here we have omitted $a$ from the notation. This can be written concisely as
\[
\d=\sum\limits_{u=1}^n\d_u
\quad\hbox{ where }\quad
(\d_uf)_I=(\Tu-\Tup)f_{uI}\eqno{(7)}
\]

\noindent{\bf The map $\delta$.} The coboundary $\delta:A^{r}\longrightarrow{}A^{r+1}$ can be defined by its action on monomials
\[\delta(fdI)=\sum\limits_{u\not\in{}I}(\Tup{}f-\Tu{}f)du\wedge{}dI\]
Since $du\wedge{}dI=0$ for $u\in{}I$, the restriction on $u\not\in{}I$ can be dropped and
\[\delta=\sum\limits_u\delta_u\quad\hbox{where}\quad
\delta_u(fdI)\equiv(\Tup{}f-\Tu{}f)du\wedge{}dI=\delta_u(f)\wedge{}dI\eqno{(8)}\]
where on the right hand side, $f$ is considered as the 0-cochain $fd\emptyset$. On general cochains,
\[
(\delta_uf)_I=\left\{\begin{array}{ll}
\pm^u_I(\Tup-\Tu)f_{I\bs{}u}&\hbox{if }u\in{}I\\
0&\hbox{if }u\not\in{}I
\end{array}
\right.\eqno{(9)}
\]
\begin{lemma} Define $\idb:A\longrightarrow{}A$ by $\idb(fdI)=(-1)^{|I|}fdI$.  With $m$, $*$, $\d$ and $\delta$ as defined above in (5), (6), (7) and (9), and $\starb=*\circ\idb$,
\begin{itemize}
\item[(i)] $\delta_u\circ\starb=*\circ\d_u$
\item[(ii)] $\d_u\circ*=\starb\circ\delta_u$
\item[(iii)] $\delta_u\circ{}m=m\circ(\delta_u\otimes\Tup+(\idb\circ{}T_u)\otimes\delta_u)$
(shifted analogue of Leibniz for $\delta$)
\end{itemize}
\end{lemma}
\begin{proof}
(i) By direct calculation
$(*(\d_uf))_I=\pm_{I^c,I}(\Tu-\Tup)f_{uI^c}$ while
\[(\delta_u(*f))_I=\left\{
\begin{array}{ll}
\pm^u_I\pm_{(I\bs{}u)^c,I\bs{}u}(\Tup-\Tu)f_{(I\bs{}u)^c}&\quad{}u\in{}I\\
0&\quad{}u\not\in{}I
\end{array}\right.
\]
so that $(\delta_u(*f))_I=(-1)^{n+1-|I|}(*\d_uf)_I$.

(ii) We obtain  $(*(\delta_uf))_I=(-1)^{n-|I|}(\d_u(*f))_I$ in a similar way to (i).
\item[(iii)] It is sufficient to verify on monomials,
\begin{align*}
\delta_u((fdI)\!\cdot\!(gdJ))
&\!=\!(\Tup-\Tu)(fg) du\wedge{}dI\wedge{}dJ\\
&\!=\!((\Tup{}f)(\Tup{}g)-(\Tu{}f)(\Tu{}g))du\wedge{}dI\wedge{}dJ\\
&\!=\!\!(\Tup{}f\!-\!\Tu{}f)du\!\wedge\!{}dI\!\cdot\!\Tup{}gdJ
\!+\!(\!-\!1)^{|I|}(\Tu{}f)dI\!\cdot\!(\Tup{}g\!-\!\Tu{}g)du\!\wedge\!{}dJ\\
&\!=\!\delta_u(fdI)\cdot\Tup(gdJ)+(-1)^{|I|}\Tu(fdI)\cdot\delta_u(gdJ)
\end{align*}
\vskip-2.4em\end{proof}
\begin{remark}
There is also a shifted analogue of Leibniz for $\delta$  (Lemma 3.1 (iii)) where the shifts are in the opposite directions, that is with $\Tu$ and $\Tup$ interchanged.
\end{remark}
\begin{remark}
Observe that in the continuum limit $h\longrightarrow0$ with $Nh$ constant, the geometry just discussed reduces to that of the torus $\T^n$ while the operators $*$, $\frac1{2h}\delta$ and $\frac1{2h}\d$ converge to the operators $*$, $d$ and $d^*$, respectively. The algebra $A$  intuitively becomes the algebra of differential forms with the usual wedge product and the discrepancy from Leibniz seen in finite differences disappears.
\end{remark}

\section{Computation of infinitesimal cumulants for $\delta$}
In this section we apply the general construction of \S2 to the cochain algebra $(A,\wedge)$ and the square-zero degree 1 map $\delta'\equiv\frac1{2h}\delta$, explicitly calculating the  infinitesimal cumulants $\delta_k$. We work in the algebra extending the scalars by the formal power series in the symbol $h$. At a given level $i$ where the scale is $2^{-i}$, the operators  written as  quotients of operators on $A$ by $2h$ or $h$, as the case may be, mean actual numerical division by the values of $h$  or $2h$ at that scale, {\it e.g.} dividing by $2h$ means dividing by $2^{-i}$.

\begin{theorem} The $k$-bracket (that is, $(k-1)^{th}$ infinitesimal cumulant) $\delta_k:S^kA\longrightarrow{}A$ of $(A,\wedge,\frac1{2h}\delta)$ can be decomposed in components
$\delta_k=\sum_{u=1}^n\delta_{k,u}$ where
\[
\delta_{k,u}=h^{k-1}m_k\circ\sum_{i=1}^k
(-\Delta'_u\circ\idb)^{\otimes(i-1)}\otimes{}\delta_{1,u}\otimes\Delta_u^{\otimes(k-i)}
\]
Here $\Delta_u\equiv\frac1h(\tu-\id)$ and $\Delta'_u\equiv\frac1h(\id-\Tu)$ are the forward and backward divided half-differences, respectively, that is,
$(\Delta_uf)_I(p)=\frac1h\big(f_I(p+h\eu)-f_I(p)\big)$ and $(\Delta'_uf)_I(p)=\frac1h\big(f_I(p)-f_I(p-h\eu)\big)$. Meanwhile $\delta_{1,u}=\frac1{2h}\delta_u$ is based on the symmetric divided difference, $\delta_{1,u}(fdI)\equiv\frac12(\Delta_u+\Delta'_u)f\cdot{}du\wedge{}dI$. Finally, $m_k$  denotes the iterated product $A^{\otimes{}k}\longrightarrow{}A$.
\end{theorem}
\begin{proof}
The proof is by induction. For $k=1$, (8) gives $\delta_{1,u}\equiv\frac1{2h}\delta_u$. For the inductive step, assume the statement of the lemma for $k-1$. By Lemma~2.7,
\begin{align*}
&\delta_{k,u}(v_1\wedge{}v_2\wedge{}v_3\wedge\ldots\wedge{}v_k)
=\delta_{k-1,u}(v_1v_2\wedge{}v_3\wedge\ldots\wedge{}v_k)\\
&\>-\!(-1)^{|v_1|}v_1\!\cdot\!{}\delta_{k-1,u}(v_2\!\wedge\!{}v_3\!\wedge\!\ldots\!\wedge\!{}v_k)
\!-\!(-1)^{|v_2|(1+|v_1|)}v_2\!\cdot\!{}\delta_{k-1,u}(v_1\!\wedge\!{}v_3\!\wedge\!\ldots\!\wedge\!{}v_k)
\end{align*}
Each term expands to a sum of $k-1$ terms indexed by $i=1,\ldots,k-1$. The contribution from $i=1$ is
\begin{align*}
&\delta_{1,u}(v_1v_2)-(-1)^{|v_1|}v_1(\delta_{1,u}v_2)
-(-1)^{|v_2|(1+|v_1|)}v_2(\delta_{1,u}v_1)\\
&=(\delta_{1,u}v_1)(T'_uv_2)+(-1)^{|v_1|}(T_uv_1)(\delta_{1,u}v_2)
-(-1)^{|v_1|}v_1(\delta_{1,u}v_2)-(\delta_{1,u}v_1)v_2\\
&=h(\delta_{1,u}v_1)(\Delta_uv_2)
-h(-1)^{|v_1|}(\Delta'_uv_1)(\delta_{1,u}v_2)
\end{align*}
times $h^{k-2}\big(m_{k-2}\circ\Delta_u^{\otimes(k-2)}\big)(v_3\otimes\cdots\otimes{}v_k)$, where in the first step Lemma 3.1(iii) was applied. This yields the $i=1$ and $i=2$ terms of the expression to be proved. The contribution from the three $i$-th terms (for $i>1$) is
\begin{align*}
&(-\Delta'_u\circ\idb)(v_1v_2)-(-1)^{|v_1|}v_1(-\Delta'_u\circ\idb)(v_2)-(-1)^{|v_2|(1+|v_1|)}v_2(-\Delta'_u\circ\idb)(v_1)\\
&=(-1)^{|v_1|+|v_2|+1}\big(\Delta'_u(v_1v_2)-v_1\Delta'_u(v_2)-\Delta'_u(v_1)v_2\big)\\
&=(-1)^{|v_1|+|v_2|+1}\tfrac1h\big(v_1v_2-(T_uv_1)(T_uv_2)-v_1(v_2-T_uv_2)-(v_1-T_uv_1)v_2\big)\\
&=(-1)^{|v_1|+|v_2|}\tfrac1h(v_1-T_uv_1)(v_2-T_uv_2)\\
&=h(-1)^{|v_1|+|v_2|}(\Delta'_uv_1)(\Delta'_uv_2)=h\big(m\circ(-\Delta'_u\circ\idb)^{\otimes2}\big)(v_1\otimes{}v_2)
\end{align*}
times $h^{k-2}m_{k-2}\big((-\Delta'_u\circ\idb)^{\otimes(i-2)}(v_3\otimes\cdots\otimes{}v_i)\otimes{}q_{1,u}v_{i+1}\otimes\Delta_u^{\otimes(k-1-i)}(v_{i+2}\otimes\cdots\otimes{}v_k)\big)$, which is the $(i+1)$-th term in the expression to be proved.
\end{proof}
\begin{remark}
In the limit $h\to0$, both differences $\Delta_u$ and $\Delta'_u$ approach the partial derivative. The factor $h^{k-1}$ in $\delta_{k,u}$ ensures that $\delta_{2,u}\to0$.  Vanishing of the first infinitesimal cumulant is precisely the statement that Leibniz holds, in other words,  the sequence of higher bracket structures, $(A,\wedge,\frac1{2h}\delta)$  approximates a  differential graded algebra, namely the differential graded algebra of differential forms $(\Omega^*,\wedge,d)$ on the torus $\T^n$.
\end{remark}

\section{Computation of infinitesimal cumulants for $\partial$}
In this section we apply the general construction of \S2 to the chain algebra $(C,\wedge)$ and the square-zero degree $-1$ boundary map $\frac1{2h}\d$, explicitly calculating the infinitesimal cumulants $\d_k$.

\noindent Define the interior product $i_u:C\longrightarrow{}C$ by $i_u(x)=*(*^{-1}x\wedge{}u)$. Equivalently, $(i_u(x))_I=x_{uI}$ while on monomials
\[
i_u(f\cdot{}I)=\left\{\begin{array}{ll}
0&\quad\hbox{if $u\not\in{}I$}\\
\pm{}f\cdot(I\backslash\{u\})&\quad\hbox{if $u\in{}I$}
\end{array}\right.
\]
where the sign is defined so that for $u\in{}I$, $u\wedge{}i_u(I)=I$.
\begin{lemma} The shifts and interior product satisfy
\begin{itemize}
\item[(i)] $(\Tu{}i_u)\circ{}m=m\circ(\Tu{}i_u\otimes{}\Tu+\Tu\idb\otimes{}\Tu{}i_u)$
\item[(ii)] $(\Tup{}i_u)\circ{}m=m\circ(\Tup{}i_u\otimes{}\Tup+\Tup\idb\otimes{}\Tup{}i_u)$
\end{itemize}
\end{lemma}
\begin{proof}
The proofs of the two parts are identical. For (i), we compute (for $u\not\in{}I$, otherwise both sides vanish),
\begin{align*}
((\Tu{}i_u)(f.g))_I&=(\Tu(f\cdot{}g))_{uI}\\
&=\sum\limits_{J\cup{}K=uI}\!\pm_{JK}^{uI}\Tu(f_Jg_K)\\
&=\sum\limits_{J\cup{}K=I}\big(\pm_{uJK}^{uI}\Tu(f_{uJ}g_K)
+\pm_{JuK}^{uI}\Tu(f_Jg_{uK})\big)\\
&=\sum\limits_{J\cup{}K=I}\!\pm_{JK}\big((\Tu{}f_{uJ})(\Tu{}g_K)
+(-1)^{|J|}(\Tu{}f_J)(\Tu{}g_{uK})\big)\\
&=\sum\limits_{J\cup{}K=I}\!\pm_{J,K}\big((\Tu{}i_u)f)_J(\Tu{}g)_K
+(-1)^{|J|}(\Tu{}f)_J((\Tu{}i_u)g)_K)\big)
\end{align*}
as required. In the third line, the sum has been split according as $u\in{}J$ or $u\in{}K$ and the sets have been relabelled.
\end{proof}
\begin{theorem} The $k$-bracket (that is, $(k-1)^{th}$ order infinitesimal cumulant) of $(C,\wedge,\frac1{2h}\d)$ can be decomposed in components
$\d_k=\sum_{u=1}^n\d_{k,u}$ where
\begin{align*}
 \d_{k,u}=-\frac12h^{k-2}&\sum_{i=1}^k
m_k\circ\Big[(\Delta_u\idb)^{\otimes(i-1)}\otimes(\Tup{}i_u)
\otimes(\Delta_u)^{\otimes(k-i)}\\
&\qquad+(-1)^k(\Delta'_u\idb)^{\otimes(i-1)}\otimes(\Tu{}i_u)
\otimes(\Delta'_u)^{\otimes(k-i)} \Big]
\end{align*}
where $\Delta_u$ and $\Delta'_u$ are the forward and backward divided half-differences, respectively, as in Theorem~4.1.
\end{theorem}
\begin{proof}
The proof proceeds by induction on $k$. For $k=1$, observe that by (7), $\d_u(f\cdot{}I)=(\Tu{}f-\tu{}f)\cdot{}i_u(I)$ and so $\d_{1,u}=\frac1{2h}\d_u=\frac1{2h}(\Tu-\Tup)\circ{}i_u$. For the inductive step, assume the statement for $k-1$. By Lemma~2.7,
\begin{align*}
&\d_{k,u}(v_1\wedge{}v_2\wedge{}v_3\wedge\ldots\wedge{}v_k)
=\d_{k-1,u}(v_1v_2\wedge{}v_3\wedge\ldots\wedge{}v_k)\\
&\,-\!(-1)^{|v_1|}v_1\!\cdot\!{}\d_{k-1,u}(v_2\!\wedge\!{}v_3\!\wedge\!\ldots\!\wedge\!{}v_k)
\!-\!(-1)^{|v_2|(1+|v_1|)}
v_2\!\cdot\!{}\d_{k-1,u}(v_1\!\wedge\!{}v_3\!\wedge\!\ldots\!\wedge\!{}v_k)
\end{align*}
Each term expands to a sum of $2(k-1)$ terms, two from each $i=1,\ldots,k-1$. The first terms for $i=1$ contribute
\begin{align*}
&(\Tup{}i_u)(v_1v_2)-(-1)^{|v_1|}v_1(\Tup{}i_uv_2)-(-1)^{|v_2|(1+|v_1|)}v_2(\Tup{}i_uv_1)\\
&=(\Tup{}i_uv_1)(\Tup{}v_2)+(\Tup\idb{}v_1)(\Tup{}i_uv_2)
-(-1)^{|v_1|}v_1(\Tup{}i_uv_2)-(\Tup{}i_uv_1)v_2\\
&=h(\Tup{}i_uv_1)(\Delta_uv_2)+h(-1)^{|v_1|}(\Delta_u\idb{}v_1)(\Tup{}i_uv_2)
\end{align*}
times $-\frac12h^{k-3}\big(m_{k-2}\circ\Delta_u^{\otimes(k-2)}\big)(v_3\otimes\cdots\otimes{}v_k)$. Here we used Lemma 5.1(ii) in the first line.  These are the evaluations of the first term in the sum for $\d_{k,u}$ to be proved, when $i=2$, $i=1$ respectively. The second terms for $i=1$ contribute similarly, only there is an extra sign, $\Tu$ replaces $\Tup$ and $\Delta'_u$ replaces $\Delta_u$.

\noindent The  contribution to $\d_{k,u}(v_1\wedge\ldots\wedge{}v_k)$ from the $i$-th terms of the second type is
\begin{align*}
&(\Delta'_u\idb)(v_1v_2)-(-1)^{|v_1|}v_1(\Delta'_u\idb)(v_2)
-(-1)^{|v_2|(|v_1|+1)}v_2(\Delta'_u\idb)(v_1)\\
&\qquad=-h\big(m\circ(\Delta'_u\idb)^{\!\otimes2}\big)(v_1\otimes{}v_2)
\end{align*}
times $-\frac12(-h)^{k-3}m\big((\Delta_u\idb)^{\otimes(i-2)}\!(v_3\otimes\cdots\otimes{}v_{i})
\otimes(\Tup{}i_uv_{i+1})\otimes
(\Delta_u)^{\otimes(k-i-1)}\!(v_{i+2}\otimes\cdots\otimes{}v_k)\big)$. The equality follows by the argument in the last block of calculation in the proof of Theorem 4.1. The resulting expression is the contribution of the second type of term for $i+1$ in the expression to be proved. A similar argument goes for the first type of term, only with a different sign and $\Delta_u$ in place of $\Delta'_u$.
\end{proof}
\begin{example}
Here are the formulae for the first few brackets for $n=3$. Write the $(k-1)^{th}$  infinitesimal cumulant $\d_k(v_1\wedge\ldots\wedge{}v_k)$ instead, using bracket notation, as $[v_1,\ldots,v_k]$. Then the  2-bracket on chains splits into pieces $[\cdot,\cdot]=[\cdot,\cdot]_x+[\cdot,\cdot]_y+[\cdot,\cdot]_z$ (this is the splitting of $\d_k$ into pieces $\d_{k,u}$)
where
 \[
    [f\cdot{}I,g\cdot{}J]_u
=\left\{
    \begin{array}{ll}
    \frac12\big((\Delta_uf+\Deltab_uf)g-f(\Delta_ug+\Deltab_ug)\big)i_u(I)\cdot{}J&\hbox{if $u\in{}I,J$}\\
    -\frac12\big((\Tu{}f)(\Deltab_ug)+(\tu{}f)(\Delta_ug)\big)\cdot{}i_u(I\cdot{}J)&\hbox{if $u\in{}I,\>u\not\in{}J$}\\
    -\frac12\big((\Deltab_uf)(\Tu{}g)+(\Delta_uf)(\tu{}g)\big)\cdot{}i_u(I\cdot{}J)&\hbox{if $u\not\in{}I,\>u\in{}J$}\\
    0&\hbox{if $u\not\in{}I,J$}
    \end{array}\right.
\]
The 3-bracket $[\cdot,\cdot,\cdot]$ on chains similarly splits into three parts as $[\cdot,\cdot,\cdot]_x+[\cdot,\cdot,\cdot]_y+[\cdot,\cdot,\cdot]_z$ where $[f\cdot{}I,g\cdot{}J,k\cdot{}K]_u$ is $\frac{h}2$ times:
\[    \begin{array}{ll}
    0&\hbox{if $u\in{}I,J,K$}\\
    (-1)^{|I|}[(f\Delta_ug-g\Delta_uf)\Delta_uk+(g\Deltab_uf-f\Deltab_ug)\Deltab_uk]\cdot{}Ii_u(J)K&\hbox{if $u\in{}I,J,K^c$}\\
    (-1)^{|I|\!+\!|J|}[(f\Delta_uk\!-\!k\Delta_uf)\Delta_ug\!+\!(k\Deltab_uf\!-\!f\Deltab_uk)\Deltab_ug]\cdot{}IJi_u(K)&\hbox{if $u\in{}I,J^c,K$}\\
    (-1)^{|I|}[(k\Delta_ug-g\Delta_uk)\Delta_uf+(g\Deltab_uk-k\Deltab_ug)\Deltab_uf]\cdot{}Ii_u(J)K&\hbox{if $u\in{}I^c,J,K$}\\
    (\Tu{}f\cdot\Deltab_ug\cdot\Deltab_uk-\tu{}f\cdot\Delta_ug\cdot\Delta_uk)\cdot{}i_u(IJK)&\hbox{if $u\in{}I,J^c,K^c$}\\
    (\Deltab_uf\cdot\Tu{}g\cdot\Deltab_uk-\Delta_uf\cdot\tu{}g\cdot\Delta_uk)\cdot{}i_u(IJK)&\hbox{if $u\in{}I^c,J,K^c$}\\
    (\Deltab_uf\cdot\Deltab_ug\cdot\Tu{}k-\Delta_uf\cdot\Delta_ug\cdot\tu{}k)\cdot{}i_u(IJK)&\hbox{if $u\in{}I^c,J^c,K$}\\
    0&\hbox{if $u\not\in{}I,J,K$}
    \end{array}
\]
\end{example}
\begin{remark}
In the limit $h\to0$, the differences $\Delta_u$ and $\Delta'_u$ approach the partial derivative. The factor $h^{k-2}$ in $\d_{k,u}$ ensures that $\d_{3,u}\to0$ although $\d_{2,u}$ does not approach zero.  Vanishing of the 3-bracket (second infinitesimal cumulant) in the continuum limit is the second order derivation property, along with a non-trivial structure $\d_2=[\d,\wedge]$, namely the Schouten-Nijenhuis bracket on  multivector fields, the natural extension of the Lie bracket on ordinary vector fields.
\end{remark}
\

\section{Binary QFT Algebras}

  We use an adaptation to our circumstances of the definition of binary QFT algebras from \cite{P} Definition 2.3.

 \medskip\begin{definition}  If a (graded) commutative associative algebra over $\Q[[h]]$ has a differential, this  determines  a structure   consisting of  all the brackets, Definition 2.4. (Known to be expressible in terms of  commutators of the  differential  with the iterated  products.) If these all vanish this is called a classical algebra. If for all $k\in\N$, the $k^{th}$ higher infinitesimal cumulant (that is, the $(k+1)$-bracket) is divisible by $h^k$, this is called a binary QFT algebra.
 \end{definition}

   \medskip\begin{definition}  The sequence of $k$-brackets on V  associated to the  structure of a binary QFT algebra  on a graded  linear space $V$  over $\Q[[h]]$ determines  a  coderivation $D$ of $SV$, of square zero.
 A coalgebra mapping $I:SV\longrightarrow{}SW$, preserving the monomial filtration which intertwines the respective coderivations $D$ on $SV$ and $D$ on $SW$, is  a (nonlinear) morphism  of binary QFT structures on V and W if its $k^{th}$ Taylor coefficient  $S^kV\longrightarrow{}W$ is divisible by $h^{k-1}$.
 \end{definition}

\medskip {\bf Discussion of $\d$ and the chain algebra.}
Take the lattice chain algebra $C$ of \S3 and extend the scalars to $\Q[[h]]$. Furthermore we add to these the difference quotient operators, $\Delta_u$ (see the statement of Theorem 4.1 for their definition), putting in the relation
\[
T'_uc-c=h\Delta_uc
\]
for all chains $c$, to make the difference between a chain and its translate by $h$ in any lattice direction formally divisible by $h$. The resulting chain algebra we will denote by $C(h)$.

 On this algebra there are unary operations which reflect the geometry of the chain algebra: $T_u$ (translation through $h$ in the $u^{th}$ lattice direction, see \S3), $\idb$ (introducing a sign $(-1)^c$ for a $c$-chain; see \S2.3), $i_u$ (interior product of the chain with the $u$-direction; see the beginning of \S5 for its definition) along with $\Delta_u$. All compositions of iterated multiplications $m_k:S^{k+1}C(h)\longrightarrow{}C(h)$ with the unary operators $T_u$, $\idb$, $i_u$, $\Delta_u$, and their linear combinations, are considered as allowed morphisms.

 Recall from the proof of Theorem 4.1 that $\Delta'_u=\Delta_u\circ{}T_u$ while from the proof of Theorem 5.2,
$\d=\sum_u(T_u-T'_u)\circ{}i_u
=-h\sum_u(\Delta_u+\Delta'_u)\circ{}i_u$.
 This makes $\d$ divisible by $h$ as an operator on $C(h)$ and so we define $\delp\equiv\frac1{h}\d$.
We can now state the formalisation of Theorem 1.2(A).
\begin{theorem}
    The triple $(C(h),\wedge,\d)$ defines a binary QFT algebra.
   \end{theorem}
\begin{proof}
   By Theorem 5.2, the $k^{th}$ infinitesimal cumulant (that is the $(k+1)$-bracket) of $(C(h),\wedge,\delp)$ is divisible by $h^{k-1}$  and so according to the above definition, $(C(h),\wedge,h\delp)$ defines a binary QFT.
 \end{proof}

\medskip{\bf Discussion of $\delta$ and the cochain algebra.}
Now we take the lattice cochain algebra $A$ of \S3 and extend the scalars to $\Q[[h]]$, adjoin difference quotients and add the relation
\[
T'_ua-a=h\Delta_ua
\]
for all cochains $a$, to make the difference between a cochain and its translate by $h$ in any lattice direction formally divisible by $h$. The resulting cochain algebra we will denote by $A(h)$.

On this algebra there are unary operations which reflect the geometry of the cochain algebra: $T_u$, $\idb$ and wedge with $du$, along with $\Delta_u$.
  Furthermore, by (8) in \S3,
  $\delta(fdI)=\sum_u(T'_uf-T_uf)du\wedge{}dI
  =h\sum_udu\wedge(\Delta_u+\Delta_u')(fdI)$
is divisible by $h$ and so we can define the operator $\deltap\equiv\frac1{2h}\delta$ on $A(h)$.

The formalisation of Theorem 1.3(A) is as follows.
\begin{theorem}
     The triple $(A(h),\wedge,\deltap)$ defines a binary QFT algebra.
\end{theorem}
\begin{proof}
  By Theorem 4.1, the $k^{th}$ infinitesimal cumulant (that is the $(k+1)$-bracket) of $(A(h),\wedge,\deltap)$ is divisible by $h^k$  as required.
   \end{proof}

\begin{remark}
 While $\delta$, $\d$ are the familiar boundary and coboundary operators of algebraic topology on chains and cochains respectively, the operators $\deltap$ and $\delp$ are the ones which will give first order discrete approximations to the continuum calculus.
 \end{remark}

\section{Maps between scales}

\subsection{Coarse scale lattice inside finer lattice} Contained in the lattice $L_h$ is the sublattice $L_{2h}=(2h\Z)^n/(4Nh\Z)^n$, a periodic cubic lattice with grid step $2h$ and size $2N$ in each direction. The vertices of $L_{2h}$ are bicoloured similarly to $L_h$ by the parity of $\frac1{2h}\sum_{u}x_u$, although it should be noted that the colourings of the two lattices are {\it not} compatible. Denote the corresponding constructions to those of \S3 on the coarser lattice by a bar, $\Ibar_a$ being a cell of type $I$ with edge lengths $4h$ and centre $a\in{}L_{2h}$, $\Ibar{}a\equiv{}IIa$ denoting a shift of $a$ forward by $2h$ in each of the directions in $I$, and so forth, culminating in the coarse chain space $\Cbar$ with its pointwise multiplication $\wedge$.

\subsection{Chain and cochain maps}
There is a natural map from a chain complex of a coarse lattice to that of a subdivision  called {\it crumbling}.  The dual map on cochains is {\it integration}, from cochains on the fine lattice to cochains on a coarser lattice.  However there is a subtlety in our case, in that our chains are generated by \textbf{all} cubes of edge length $2h$, so that there are finer $k$-cubes which do not appear in subdivisions of coarser $k$-cubes.

 To analyze this for  $n=3$, the chain complexes we are studying  at each scale are naturally  the direct sum of 8 copies of  shift-isomorphic subcomplexes.  Cutting  the scale $h$  in half also yields 8 direct summands of  the finer chain complex.
The finer 8 can be respectively shifted to become subdivisions of the coarser 8 summands and then the crumbling chain mappings (and their dual, the integration maps on cochains) are  present. However, the  shifts that effect this picture are not unique.

To give precise formulae, note that $C=E^{\otimes{}n}$ where $E$ is the chain complex of a one-dimensional lattice $h\Z/4Nh\Z$, whose chains are generated by points $\emptyset_a$ and intervals of length $2h$, the one centred at $a$ being denoted $x_a$. On $E$, the boundary acts $\d:E_1\longrightarrow{}E_0$,
\[
\d(x_a)=\emptyset_{a+h}-\emptyset_{a-h}
\]
In this one-dimensional setting, the chain complex splits into a direct sum of two copies, those with even vertices and those with odd vertices. Choose the shifts on the finer chain complex so as to be trivial on the even vertex subcomplex, and to be a shift to the right by $h$ on the odd vertex subcomplex. The corresponding  crumbling map $\iota:\Ebar\longrightarrow{}E$ is given by
\[
\iota(\emptyb_a)=\left\{\begin{array}{ll}
\hskip-1.5ex\emptyset_a&\hskip-1ex\hbox{if }a\in4h\Z\\
\hskip-1.5ex\emptyset_{a-h}&\hskip-1ex\hbox{if }a\in2h+4h\Z
\end{array}\right.\hbox{and }
\iota(\xbar_a)=\left\{\begin{array}{ll}
\hskip-1.5ex{}x_a+x_{a-2h}&\hskip-1ex\hbox{if }a\in4h\Z\\
\hskip-1.5ex{}x_{a-h}+x_{a+h}&\hskip-1ex\hbox{if }a\in2h+4h\Z
\end{array}\right.
 \]
\begin{lemma} $\iota:\Ebar\longrightarrow{}E$ is a chain map for $\d$, that is  $\iota\circ\db=\d\circ\iota$.
\end{lemma}
Tensoring $n$ times, $\iota$ induces a chain map $\iota:(\Cbar,\db)\longrightarrow(C,\d)$.

\medskip\noindent Similarly for cochains, identify the space of cochains on the $n$-dimensional lattice as $A=B^{\otimes{}n}$ where $B$ is the cochain complex of the one-dimensional lattice, consisting of 0-cochains $f(x)$ and 1-cochains $f(x)dx$ (which takes value $f(x)$ on the interval whose centre is $x$), using the notation of \S3. On $B$, the coboundary acts $\delta:B_0\longrightarrow{}B_1$ with
\[
\delta{}f=\left(f(x+h)-f(x-h)\right)dx
\]
The integration map $\Bar:B\longrightarrow{}\Bbar$ dual to the above crumbling map on chains is given on 0-cochains by
\[
\overline{f}(x)=\left\{\begin{array}{ll}
\hskip-1.5ex{}f(x)&\hskip-1ex\hbox{if }x\in4h\Z\\
\hskip-1.5ex{}f(x-h)&\hskip-1ex\hbox{if }x\in{}2h+4h\Z
\end{array}
\right.
\]
while on 1-cochains, $f(x)dx\longmapsto{}g(x)dx$ where
\[
g(x)=\left\{\begin{array}{ll}
\hskip-1.5ex\half\left(f(x)+f(x-2h)\right)&\hskip-1ex\hbox{if }x\in4h\Z\\[2pt]
\hskip-1.5ex\half\left(f(x-h)+f(x+h)\right)&\hskip-1ex\hbox{if }x\in{}2h+4h\Z
\end{array}
\right.\]
Tensoring $n$ times gives a cochain map $\Bar:(A,h^{-1}\delta)\longrightarrow(\Abar,(2h)^{-1}\deltabar)$, that is $\half\deltabar\circ\Bar=\Bar\circ\delta$.

\subsection{General theory of morphisms extending chain maps}
Suppose now that $(V,m)$ and $(\Vbar,\mbar)$ are two commutative associative algebras, each endowed with a square-zero map of (the same) grading $\pm1$, $\d:V\longrightarrow{}V$ and $\dbar:\Vbar\longrightarrow{}\Vbar$. Suppose that $\Bar:V\longrightarrow\Vbar$ is a chain map respecting the grading, that is $\dbar\overline{x}=\overline{(\d{}x)}$ for $x\in{}V$. Note that there is no assumption about the compatibility or otherwise of $\Bar$ with the algebra structures on $V$ and $\Vbar$.

\medskip\noindent Following the construction of Theorem 2.10, we generate coderivations $D$ and $\Dbar$ on $S^*V$ and $S^*\Vbar$ respectively, each of square-zero, whose Taylor coefficients describe the infinitesimal cumulants of $\d$ and $\dbar$, with respect to $m$ and $\mbar$, respectively. Let $S(\Bar):S^*V\longrightarrow{}S^*\Vbar$ be the image of $\Bar$ under the functor $S$ from chain complexes to differential coalgebras; that is
\[S(\Bar):\ v_1\wedge\ldots\wedge{}v_k\longmapsto\vbar_1\wedge\ldots\wedge\vbar_k\]
This is a coalgebra map and it follows from the fact that $\Bar$ is a chain map, that
\[S(\Bar)\circ{}\d^\wedge=\dbar^\wedge\circ{}S(\Bar)\]
Using the cumulant bijections $\tau:S^*V\longrightarrow{}S^*V$ and $\taubar:S^*\Vbar\longrightarrow{}S^*\Vbar$ which are coalgebra maps, the composition
\[\sigma\equiv\taubar^{-1}\circ{}S(\Bar)\circ\tau:\ S^*V\longrightarrow{}S^*\Vbar\]
is seen to be a coalgebra map of degree $0$.

\medskip\begin{lemma} The map $\sigma:S^*V\longrightarrow{}S^*\Vbar$ defines a coalgebra morphism with $\sigma\circ{}D=\Dbar\circ\sigma$.
\end{lemma}
The proof is the commuting cube diagram below.
\[\includegraphics[width=.3\textwidth]{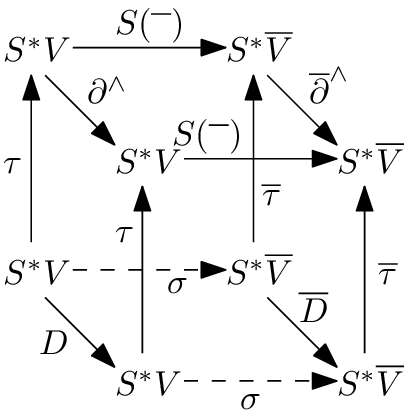}
\]
As with any coalgebra map, $\sigma$ is determined by its Taylor coefficients, which we denote $\sigma_k:S^kV\longrightarrow\Vbar$. We will now give an explicit formula for the action of $\sigma_k$ along with a recursion relation, which will use in the proof of estimates in \S7.4. These will be analogues of Lemma 2.6 and Lemma 2.7.

\begin{lemma}
\[\sigma_k(v_1\wedge\cdots\wedge{}v_k)
=\sum\limits_{r=1}^k(-1)^{r-1}(r-1)!
\sum\limits_{I_1\cup\cdots\cup{}I_r=[k]}\pm\taubar_1\Big(\overline{\tau_1(v_{I_1})}\wedge\cdots\wedge\overline{\tau_1(v_{I_r})}\Big)\]
where the inner summation is over all (unordered) partitions of $[k]$ into $r$ subsets, the sign is that determined by the change of order of symbols $v_1\wedge\cdots\wedge{}v_k
=\pm{}v_{I_1}\wedge\cdots\wedge{}v_{I_r}$.
\end{lemma}
\begin{proof}
By definition, $\sigma=\taubar^{-1}\circ{}S(\Bar)\circ\tau$. Its $k^{th}$ Taylor coefficient is $\sigma_k=(p_1\circ\taubar^{-1}\circ{}S(\Bar)\circ\tau)|_{S^kV}$
which we compute from the first part of Lemma 2.3,
\[
\sigma_k(v_1\wedge\cdots\wedge{}v_k)
=(p_1\circ\taubar^{-1}\circ{}S(\Bar))\sum\limits_{r=1}^k
\sum\limits_{I_1\cup\cdots\cup{}I_r=[k]}
\pm\tau_1(v_{I_1})\wedge\cdots\wedge\tau_1(v_{I_r})
\]
where the sum is over all (unordered) partitions of $[k]$ into $r$ non-empty subsets, and the sign is as in the statement of the lemma. Recalling the definition of $S(\Bar)$ and applying the equality $p_1\circ\taubar^{-1}=d_r\cdot\taubar_1$ on $S^rV$ (Lemma 2.3), the statement of the lemma follows.
\end{proof}
\begin{remark}
The last lemma identifies $\sigma_k$ with the $k^{th}$ (commutative) cumulant of $\Bar$ with respect to the multiplications $m$ and $\mbar$; see \cite{R2}.  The first few are
\begin{align*}
    \sigma_1(u)&=\ubar\\
    \sigma_2(u\wedge{}v)&=\overline{uv}-\ubar\cdot\vbar\\
    \sigma_3(u\wedge{}v\wedge{}w)&=\overline{uvw}-\overline{uv}\cdot\wbar-\ubar\cdot\overline{vw}-(-1)^{|u|\cdot|v|}\vbar\cdot\overline{uw}+2\ubar\cdot\vbar\cdot\wbar
    \end{align*}
\end{remark}

\begin{lemma} For $u,v\in{}V$ and $w\in{}S^kV$,
\[\sigma_{k+2}(u\wedge{}v\wedge{}w)
=\sigma_{k+1}(uv\wedge{}w)
-\sum_{J\subseteq[k]}\pm\sigma_{|J|+1}(u\wedge{}w_J)
\sigma_{|J^c|+1}(v\wedge{}w_{J^c})
\]
where the sum is over all (not necessarily proper) subsets $J$ of $[k]$ and the sign is so that $w_J\wedge{}v\wedge{}w_{J^c}=\pm{}v\wedge{}w$.
\end{lemma}
\begin{proof}
By Lemma 7.3, the left hand side can be written as a sum of terms indexed by partitions of  $\{u,v\}\cup[k]$.

If $u$ and $v$ are in the same set in the partition then the corresponding term in $\sigma_{k+2}(u\wedge{}v\wedge{}w)$ will also appear in $\sigma_{k+1}(uv\wedge{}w)$ and with the same coefficient.

On the other hand, if $u$ and $v$ are in different sets in the partition $I_1\cup\cdots\cup{}I_r$, say $u\in{}I_1$ and $v\in{}I_2$, then the corresponding term in the LHS of the equality in the lemma may appear several times in the sum on the RHS. In particular, it will appear in the term on the RHS labelled by the set $J\subseteq[k]$ if $\{u\}\cup{}J$ can be written as a union of some collection of the sets $I_j$ (while $\{v\}\cup{}J^c$ is the union of the complementary collection), say $\cup_{j\in{}K}I_j$ for some $K\subseteq[r]$ where $1\in{}K$ and $2\not\in{}K$. The coefficient of the term is
\[
-\sum_{K\subseteq[r],1\in{}K,2\not\in{}K}d_{|K|}d_{r-|K|}
=(-1)^{r-1}\sum_{m=1}^{r-1}
{{r-2}\choose{m-1}}(m-1)!(r-m-1)!=d_r\]
which matches the coefficient of the same term in $\sigma_{k+2}(u\wedge{}v\wedge{}w)$
\end{proof}

\begin{lemma}
Suppose that $V$, $W$, $\Vbar$ and $\Wbar$ are (graded) commutative associative algebras with given grading preserving maps $\Bar:V\longrightarrow\Vbar$ and $\Bar:W\longrightarrow\Wbar$. Then their tensor product $\Bar:V\otimes{}W\longrightarrow\Vbar\otimes\Wbar$ has $k^{th}$ commutative cummulant $S^k(V\otimes{}W)\longrightarrow\Vbar\otimes\Wbar$ given by
\begin{align*}
&(v_1\otimes{}w_1)\wedge\cdots\wedge(v_k\otimes{}w_k)\\
&\longmapsto\sum_{r=1}^k\sum_{I_1\cup\cdots\cup{}I_r=[k]}\pm\sigma^V_{|I_1|}(v_{I_1})\cdots\sigma^V_{|I_r|}(v_{I_r})\otimes\sigma^W_r(\tau_1(w_{I_1})\wedge\cdots\wedge\tau_1(w_{I_r}))
\end{align*}
in terms of the cummulants $\sigma^V_k:S^kV\longrightarrow\Vbar$ and $\sigma^W_k:S^kW\longrightarrow\Wbar$. Here the inner sum is over all unordered partitions of $[k]$ into $r$ non-empty sets and the sign is that induced by the change of order
\[v_1\wedge{}w_1\wedge\cdots\wedge{}v_k\wedge{}w_k=\pm{}v_{I_1}\wedge\cdots\wedge{}v_{I_r}\wedge{}w_{I_1}\wedge\cdots\wedge{}w_{I_r}\]
\end{lemma}
\begin{proof}
Unordered partitions $\pi$ of $[k]$ are in bijection with equivalence relations $\sim$ on the set $[k]$. There is a partial ordering on equivalence relations by which $\sim'$ is said to be {\it stronger} than $\sim$ if whenever $a\sim'b$ then also $a\sim{}b$. The corresponding partial ordering on (unordered) partitions will be written $\pi'\geq\pi$ and we say that $\pi'$ is a {\it finer} partition than $\pi$ (equivalently, $\pi$ is said to be {\it coarser} than $\pi'$, written $\pi\leq\pi'$).

For a partition $\pi$ of $[k]$ represented by $I_1\cup\cdots\cup{}I_r=[k]$, denote by $|\pi|$, the number of equivalence classes (that is, $r$). Denote by $v_\pi\in{}V$, the element $\taubar_1\big(\tau_1(v_{I_1})\wedge\cdots\wedge\tau_1(v_{I_r})\big)$.
In this notation, Lemma 7.3 can be written,
\[\sigma_k(v_1\wedge\cdots\wedge{}v_k)=\sum_{\pi}d_{|\pi|}v_\pi\]
Using Lemma 7.3 to expand out the right hand side of the expression to be proved, we obtain
\[
\sum_\pi\sum_{\pi'\geq\pi}\pm\Big(\prod_{j=1}^{|\pi|}{d_{\alpha_j}}\Big)v_{\pi'}\otimes\sum_{\pi''\leq\pi}d_{|\pi''|}w_{\pi''}\eqno{(*)}
\]
where the sums are over all partitions $\pi$ of $[k]$, finer partitions $\pi'$ and coarser partitions $\pi''$ and $\alpha_j$ ($1\leq{}j\leq|\pi|$) is the number of sets into which $\pi'$ splits the $j^{th}$ set of $\pi$.  The coefficient of a term $v_{\pi'}\otimes{}w_{\pi''}$ for partitions $\pi'$ and $\pi''$ of $[k]$ with $\pi'\geq\pi''$ is
\[
\pm{}d_{|\pi''|}\sum_{\pi:\pi''\leq{}\pi\leq\pi'}\prod_{j=1}^{|\pi|}{d_{\alpha_j}}
\]
where the sign is given by the change of order of $v$'s and $w$'s (which is independent of $\pi$).  If $\beta_i$ ($1\leq{}i\leq|\pi''|$) denotes the number of sets into which $\pi'$ splits the $i^{th}$ set of $\pi''$ then $\pi$ is determined by a choice of partition of $[\beta_i]$ for each $i$. The above sum then splits as a product,
\[
\pm{}d_{|\pi''|}\prod_{i=1}^{|\pi''|}\Big(\sum_{J_1\cup\cdots\cup{}J_s=[\beta_i]}\prod_{l=1}^{s}{d_{|J_l|}}\Big)
\]
By the last statement in the proof of Lemma 2.3, the inner sum vanishes unless $\beta_i=1$ and then it is $1$. The whole expression therefore vanishes unless $\beta_i=1$ for all $i$, that is, $\pi'=\pi''$ and then is $\pm{}d_{|\pi''|}$. The expression (*) therefore reduces to
\[
\sum_{\pi''}\pm{}d_{|\pi''|}v_{\pi''}\otimes{}w_{\pi''}
\]
which is, by Lemma 7.3, the $k^{th}$ cumulant of the tensor product, as required.\end{proof}

\subsection{Application to refinement of cubical lattices}
The construction of \S7.3 now induces from the cochain map $\Bar$ of \S7.2, a corresponding extended morphism from fine to coarse, which intertwines the coderivations $D$ and $\Dbar$ on the symmetric cofree coalgebras of cochains. That is, they `intertwine' the collection of infinitesimal cumulants from the finer and coarser lattices, although the exact relations between individual coarse and fine infinitesimal
cumulants are not so simple, see Remark 7.9 below.

\medskip\begin{lemma}
For $k>1$, the cumulant $\sigma_k:S^kB\longrightarrow\Bbar$ of the map $\Bar$ defined in \S7.2 on the one-dimensional lattice, is divisible by $h^{k-1}$, that is, the components of the cochain $\sigma_k(v_1\wedge\cdots\wedge{}v_k)\in\Bbar$ (with respect to the natural basis) can be written as sums of terms, each of which contains a product of at least $k-1$ factors of the form of a difference in values of components of the cochains $v_1,\ldots,v_k$ at lattice points $2h$ apart.
\end{lemma}
\begin{proof}
  By Lemma 7.3,
  \[\sigma_k(v_1\wedge\cdots\wedge{}v_k)
=\sum\limits_{r=1}^k\tfrac{(-1)^{r-1}}{r}
\sum\limits_{I_1\cup\cdots\cup{}I_r=[k]}\pm\taubar_1\Big(\overline{\tau_1(v_{I_1})}\wedge\cdots\wedge\overline{\tau_1(v_{I_r})}\Big)\]
where the sum is over ordered partitions $I_1\cup\cdots\cup{}I_r$ of $[k]$ into $r$ non-empty subsets. Now $\Bbar$ has a basis consisting of 0-cochains localized at a single point, $\chib_a$, and 1-cochains localized on a single interval (centre $a$) of length $4h$, $\chib_adx$, where $a$ is a multiple of $2h$. Since the formulae for $\Bar$ are invariant under shifts by $4h$, it suffices to compute the coefficients of $\chib_0$, $\chib_{2h}$, $\chib_0dx$ and $\chib_{2h}dx$. For example, to compute the component of $\chib_0$ in $\sigma_k(v_1\wedge\cdots\wedge{}v_k)$, note that
\[
[\taubar_1(w_1\wedge\cdots\wedge{}w_r)]_{\chib_0}\!=\!\prod_{j=1}^r[w_j]_{\chib_0},\quad
[\vbar]_{\chib_0}=[v]_{\chi_0},\quad
[\tau_1(v_1\wedge\cdots\wedge{}v_s)]_{\chi_0}\!=\!\prod_{j=1}^s[v_j]_{\chi_0}
\]
from which we obtain the following expression for $[\sigma_k(v_1\wedge\cdots\wedge{}v_k)]_{\chib_0}$,
\[
\sum\limits_{r=1}^k\tfrac{(-1)^{r-1}}{r}
\sum\limits_{I_1\cup\cdots\cup{}I_r=[k]}\prod_{j=1}^r\Big(\prod_{i\in{}I_j}[v_i]_{\chi_0}\Big)
=\Big(\prod_{i=1}^k[v_i]_{\chi_0}\Big)\cdot\sum_{r=1}^k\frac{(-1)^{r-1}}rN_{k,r}=0\eqno{(10)}
  \]
the latter equality following by a proof similar to that of (3) in \S2.3. Similarly $[\sigma_k(v_1\wedge\cdots\wedge{}v_k)]_{\chib_{2h}}$ vanishes for $k>1$.   To complete the proof of the lemma, we need to check the divisibility by $h^{k-1}$ of components of $\chib_0dx$ and $\chib_{2h}dx$ in $\sigma_k(v_1\wedge\cdots\wedge{}v_k)$. Note that
\begin{align*}
[\taubar_1(w_1\wedge\cdots\wedge{}w_r)]_{\chib_0dx}
&=\sum_{p=1}^r\Big([w_p]_{\chib_0dx}\prod_{j=1\atop{}j\not=p}^r[w_j]_{\chib_0}\Big)\\
[\vbar]_{\chib_0}=[v]_{\chi_0},\quad [\vbar]_{\chib_0dx}&=\tfrac12[v]_{\chi_0dx}+\tfrac12[v]_{\chi_{-2h}dx}\\
[\tau_1(v_1\wedge\cdots\wedge{}v_s)]_{\chi_0}\!=\!\prod_{j=1}^s[v_j]_{\chi_0},\,\,\,
&[\tau_1(v_1\wedge\cdots\wedge{}v_r)]_{\chi_adx}
\!=\!\sum_{p=1}^r\Big([v_p]_{\chi_adx}\prod_{j=1\atop{}j\not=p}^r[v_j]_{\chi_a}\Big)
\end{align*}
from which we obtain the following expression for $[\sigma_k(v_1\wedge\cdots\wedge{}v_k)]_{\chib_0dx}$,
\begin{align*}
\sum\limits_{r=1}^k\tfrac{(-1)^{r-1}}{r}
&\!\!\!\sum\limits_{I_1\cup\cdots\cup{}I_r=[k]}
\sum_{p=1}^r\frac12\prod_{j=1\atop{j\not=p}}^r\Big(\prod_{i\in{}I_j}[v_i]_{\chi_0}\Big)\\
&\cdot\Big(\sum_{q\in{}I_p}\big([v_q]_{\chi_0dx}\prod_{i\in{}I_p\bs\{q\}}[v_i]_{\chi_0}+[v_q]_{\chi_{-2h}dx}\prod_{i\in{}I_p\bs\{q\}}[v_i]_{\chi_{-2h}}\big)\Big)\quad(11)
\end{align*}
This is totally symmetric under permutation of the $v_i$'s and depends linearly on the coefficients of the localized 1-cochains $\chi_0dx$ and $\chi_{-2h}dx$ in $v_1,\cdots,v_k$. In particular, the coefficient of $[v_k]_{\chi_0dx}$ in (11) vanishes, reducing to a sum of a form similar to (10). Meanwhile, the coefficient of $[v_k]_{\chi_{-2h}dx}$ in (11) is
\[
\sum\limits_{r=1}^k\!\tfrac{(-1)^{r-1}}{2r}\!\sum_{p=1}^r
\!\sum\limits_{I_1\cup\cdots\cup{}I_r=[k]\atop{}k\in{}I_p}\prod_{i\in{}I_p\bs{}k}[v_i]_{\chi_{-2h}}\cdot\!\!\!
\prod_{i\in{}\cup_{j\not=p}I_j}[v_i]_{\chi_0}
\]
Noting that the internal summand depends on the partition $\{I_j\}$ only via $I_p$, we can rewrite it as $\sum_{r=1}^k\half(-1)^{r-1}\!\sum_{k\in{}I\subseteq[k]}N_{k-|I|,r-1}\!
\!\prod_{i\in{}I\bs{}k}[v_i]_{\chi_{-2h}}\cdot
\!\prod_{i\in{}I^c}[v_i]_{\chi_0}$, which can be identified as $\tfrac12\big(\frac{\partial}{\partial[v_k]_{\chi_{-2h}}}\big)$ of the expression
\[
\sum\limits_{r=1}^k(-1)^{r-1}
\sum_{I\subseteq[k]}N_{k-|I|,r-1}
\prod_{i\in{}I}[v_i]_{\chi_{-2h}}\cdot
\prod_{i\in{}I^c}[v_i]_{\chi_0}\eqno{(12)}
  \]
Splitting the sum over $I$ according to order and recalling from \S2.3, (3) that $\sum_{r=1}^{k}(-1)^{r-1}N_{k-s,r-1}=(-1)^{k-s}$, the expression (12) reduces to
\[\sum_{s=0}^k(-1)^{k-s}\,\sum_{I\subseteq[k],|I|=s}\,
\prod_{i\in{}I}[v_i]_{\chi_{-2h}}\cdot
\prod_{i\in{}I^c}[v_i]_{\chi_0}
=\prod_{i=1}^k([v_i]_{\chi_{-2h}}-[v_i]_{\chi_0})\]
Thus, its (partial) derivative as used above, is a product of $(k-1)$ differences of the form $[v_i]_{\chi_{-2h}}-[v_i]_{\chi_0}$, $i\not=k$, and consequently $[\sigma_k(v_1\wedge\cdots\wedge{}v_k)]_{\chib_0dx}$ can be written as a sum of terms, each of which is a product of $k$ factors, $k-1$ of which being differences of this form; hence it is divisible by $h^{k-1}$.

The proof that $[\sigma_k(v_1\wedge\cdots\wedge{}v_k)]_{\chib_{2h}dx}$ is divisible by $h^{k-1}$ is identical to the previous proof, with $\chi_0$, $\chi_{-2h}$ replaced by $\chi_h$, $\chi_{3h}$.
\end{proof}

\begin{lemma}
Extending $\Bar$ from the cochains $B$ on the  dimension one lattice to the cochains $A$ on the  dimension $n$ lattice by tensor product, the  cummulant $\sigma_k:S^kA\longrightarrow\Abar$ of the map $\Bar$ so defined is divisible by $h^{k-1}$ for $k>1$.
\end{lemma}
\begin{proof}
By Lemma 7.6, the tensor product of any two structures for which the $k^{th}$ cumulants are divisible by $h^{k-1}$ will also satisfy this property, since the term in the sum in Lemma 7.6 indexed by the partition $I_1\cup\cdots\cup{}I_r$ will have $h$-order,
\[
\sum_{j=1}^r(|I_j|-1)+(r-1)=k-1
\]
Combining with Lemma 7.7, the result follows.
\end{proof}

\begin{remark}
The conclusion is that for adjacent scales, the integration map $\Bar$ between them $A\longrightarrow\Abar$, provides an induced map (given by composition with cumulant bijections and whose Taylor coefficients are identified by Lemma 7.3 with commutative cumulants) $\sigma:S^*A\longrightarrow{}S^*\Abar$ which satisfies $\sigma\circ{}D=\Dbar\circ\sigma$ as maps (coderivations) $S^*A\longrightarrow{}S^*\Abar$. The equality of Taylor coefficients of the maps on the two sides of this equation gives the exact relation between the systems of $k$-brackets (infinitesimal cumulants) at the coarse and fine scales, namely a collection of relations,
\begin{align*}
    [\sigma_1,\delta'_1]&=0\\
    [\sigma_1,\delta'_2]+[\sigma_2,\delta'_1]&=0\\
    \cdots&\\
    [\sigma_1,\delta'_k]+[\sigma_2,\delta'_{k-1}]+\cdots+[\sigma_k,\delta'_1]&=0
\end{align*}
 where the $k^{th}$ equation is an equality between maps $S^kV\longrightarrow\Vbar$ ($\sigma$'s extended as coalgebra maps and $\delta$'s extended as coderivations). Here an extended notion of commutators is being used whereby for a map $b:S^*V\longrightarrow{}S^*\Vbar$ we denote by $[b,x]$ the difference $b\circ{}x-\overline{x}\circ{}b$, which is a map $S^*V\longrightarrow{}S^*\Vbar$, for any symbol $x$ which represents a map $S^*V\longrightarrow{}S^*V$ and for which there is an analogous map on the `bar'-side $\xbar:S^*\Vbar\longrightarrow{}S^*\Vbar$.

The first equation is just the statement that $\Bar$ is a cochain map. The second relation in long-hand states that
\[\overline{[v,w]}-[\vbar,\wbar]
=\deltabar'\circ\sigma_2\big(v\wedge{}w)-\sigma_2(\delta'{}v\wedge{}w+(-1)^{|v|}v\wedge\delta'{}w\big)
\]
\end{remark}

The above collection of relations between the $k$-brackets $\delta'_k$ and $\deltabar'_k$ (which are of order $h^{k-1}$ for $k>1$, by Theorem 4.1) defines the collection of maps $\sigma_k:S^kA\longrightarrow\Abar$ (of order $h^{k-1}$ for $k>1$ by Lemma 7.8) to be a morphism (in the sense of Definition 6.2) between the two binary QFT algebras $A$ and $\Abar$ (of Theorem 6.4). Hence we obtain a formulation of Theorem 1.3(B) as follows.

\begin{theorem}
Considering the binary QFT algebra $(A(h),\wedge,\delta')$ of Theorem 6.4 for scales $h=h_0\cdot2^{-m}$, $m\in\N$, gives rise to an inverse system of  binary QFT algebras  related by  binary QFT morphisms between the structures at the various scales.
\end{theorem}

The chain map $\iota$ of \S7.2 from the coarse chain complex $(\Cbar,\dbar)$ to the fine chain complex $(C,\d)$, will according to \S7.3 induce maps $S^*\Cbar\longrightarrow{}S^*C$ which intertwine the infinitesimal cumulants at the coarse and fine levels of $(C,\d)$ of Theorem 5.2. The following theorem is deduced by exactly similar arguments to those in the proof of Theorem 7.10, based on direct computation of the  cumulants of the explicit chain map using Lemmas 7.3 and 7.6.

\begin{theorem}
Considering the binary QFT algebra $(C(h),\wedge,\d)$ of Theorem 6.3 for scales $h=h_0\cdot2^{-m}$, $m\in\N$, gives rise to  a direct system of binary QFT algebras related by binary QFT morphisms between the structures at the various scales.
\end{theorem}

\bigskip{\bf Background Appendix}

 At the continuum level there are two product  structures on  distinct but isomorphic spaces which  play different roles. The one on differential forms is familiar in  topology because of integration and  mapping properties.  The  other on multivector fields is  familiar in differential geometry. There is a third contraction product
between a $k$-differential form and a $j$-multivector field  giving a $(j-k)$-multivector field if $j>k$,
a $(k-j)$-differential form if $k>j$ and a function if $j=k$ which by definition is both a differential form  and a  multivector field.

The  $j>k$  case of the contraction product has an interpretation in algebraic topology, {\it e.g.}\ Whitney's cap product, while the $k>j$ case is used often in differential geometry, {\it e.g.} in Cartan's magic formula.

\medskip\textbf{One algebra structure}

 The integration of  a $k$-differential form  over an oriented $k$-surface  defines a $k$-cochain, {\it i.e.} a linear function on $k$-chains, the term used  for  linear combinations of oriented $k$-surfaces.  There is the geometric boundary map on chains called $\d$, taking a $(k+1)$-chain to a $k$-chain,  and thus a dual operator on cochains referred to here as $\delta$, taking a $k$-cochain to a $(k+1)$-cochain. Poincar\'e calculated and defined the general exterior $d$ on differential forms and showed $dd=0$. Stokes' theorem says $\int:(k-\hbox{forms})\longrightarrow(k-\hbox{cochains})$   intertwines  Poincare's  exterior  $d$  acting on forms or integrands  with $\delta$ acting on cochains. This was Poincar\'e's picture {\it c.}1900 and marked the birth of algebraic topology.

    Using exterior or Grassmann algebra, differential forms have a fully defined  graded commutative and associative algebra structure and a differential $d$ of square zero which  is a derivation of this product.  It is  of great advantage that  this  differential algebraic structure on forms  is natural for all appropriate mappings between manifolds. See \cite{PS}, \cite{S77} and \cite{W}.

      One knows in topology that it is impossible to  break  space into  finite  cells and to construct a finite-dimensional algebraic structure  on the  cochains  that reflects perfectly the properties of $d$ and this product on differential forms. Commutativity, associativity and Leibniz cannot all be attained in a finite-dimensional setting (preserving homology) and so the fact that these do all hold for differential forms has forced that space of such to be infinite-dimensional.

  \break\textbf{The other algebra structure}

  Secondly, at the continuum level,  there is the exterior product  structure on the multivector fields, by definition the cross sections of the exterior powers of the tangent bundle. This product is similar in appearance to the  graded algebra structure  of differential forms. But it has a quite different
   meaning.

 Given a volume  measure on the ambient manifold, $k$-multivector fields via contraction  with $k$-differential forms followed by integration of the resulting function, define elements in the dual space  of $k$-differential forms. There is thus defined a dual operator  of degree $-1$,  call it  $\d$,  defined by  linear duality from the exterior derivative $d$ on forms.  This $\d$ is consistent with the  geometric  boundary operator $\d$ on chains  mentioned above. Thus the multivector fields have a  graded commutative associative product and a square zero operator of degree $-1$, in the presence of a volume  measure. This structure is natural for diffeomorphisms that preserve volume.

  For a manifold of dimension $d$, contracting multivector fields with a fixed volume form yields an isomorphism between
multivector fields and  differential forms which relates  $\partial$ to
 exterior $d$, providing a  geometric version of Poincar\'e duality.

    There is a difference between product rules for the two structures on differential forms and on  multi-vector currents that arises because $d$ is a derivation of its product,  the exterior product of forms,  while  $\partial$ is not a first order derivation.

   A bracket operation on multivector fields is defined as the deviation of $\d$ from being a  derivation of the  exterior product of multi-vector fields. By construction, $\d$
   is a derivation of this bracket. It turns out, again remarkably, that this deviation bracket is itself  a derivation of the  exterior product in one of its arguments with the other held fixed. This is called the \textbf {\sl second order derivation property} of the continuum $\d$ relative to its product.  It is noteworthy that this second order property  plus $\d^2=0$, implies the Jacobi identity holds for this bracket of multi-vector fields.

   One may note that graded commutative associative algebras with a differential which is a second order derivation were studied in a beautiful paper \cite{Koszul} and eventually dubbed $BV$ algebras because of the perturbative quantum theory algorithms given in \cite{BV}. This $BV$ perspective at the finite level is potentially useful. However, here one lands instead on the binary QFT algebras described in \S6 which are more general and work fine.

   \medskip
\textbf{Metric or symplectic structures}

If the volume form is that of a  Riemann metric which is used to identify  multivector fields and forms  of the same degree and  thus also their respective products, then $\d$  becomes the metric  adjoint of exterior $d$. Recall that this adjoint is the conjugate of exterior $d$ by the Hodge star operator.  It is often  denoted $d^*$  or $\d$, has degree $-1$ and  its square is zero.
  \indent One sees  a significant  universal identity between $d$,  the exterior product  and  the Hodge star operator on forms  for all Riemannian manifolds:
  ``the adjoint of exterior $d$ is a second order derivation of the exterior product of forms".

   If the identification between tangent space and cotangent space is defined by a closed  nondegenerate two form  there is an analogous second order operator of square zero
   leading to  the the physicist's Batalin-Vilkovisky  calculations  motivating the  definition in \cite{P}  of a binary QFT algebra. The latter  is realized at the discrete level  above while the former requires the continuum.

\textbf{Funding:}
This work was supported by the US-Israel Binational Science Foundation [grant 2016219 to R.L.\&D.S.]; Balzan   Foundation [2014 Prize to D.S].

\bigskip{\bf  Postscript: Sir Michael Atiyah}

In August 1966 the third author, age 25, on the road between Warwick and Southampton to board a ship home, was graciously invited to lunch by Professor Michael Atiyah  at his  college  in Oxford to discuss math  questions related to $K$-theory  orientations for PL  manifolds, posed the evening before on the telephone.  On the way in to dining,  some  college fellows were congratulating  Professor Atiyah on his just awarded Fields Medal received in Moscow.

Professor Atiyah offered advice  during a jolly  lunch about good, better and best fields of Mathematics.  It was amazing how he could make conceptual connections between different fields. Still, young people could resist the advice of their elders: PL manifolds,  good; smooth manifolds-differential geometry, better; analysis-algebraic geometry-number theory, best.

  At the 1968 Global Analysis  Summer Conference in Berkeley, Michael Atiyah described, during a swim in the  Strawberry Canyon pool,  ``triality and the six sphere"  plus  a fascinating  way to understand, via complex conjugation, the holomorphic structure on the real Grassmanian of oriented two-planes in $\R^n$.

  In 1969, Michael Atiyah, at MIT's cafeteria related  Hironaka's resolution of singularities to  Hormander's characteristic variety   for hyperbolic PDE's.

  During  Professor Atiyah's tenure {\it c.}1970 at the IAS in Princeton,  this writer showed up at his office door very early one morning  with a zany idea.
  Atiyah managed to relate  Adams operations and    K-theory of PL manifolds to Brumer's 1967 linear independence over $\Q$ of $p$-adic logarithms of algebraic numbers.

  Michael Atiyah was also  telling in those days of the early 70's, of a very neat way to think about proving groups of odd order were solvable. This, by finding a fixed point theorem  for actions of such groups on complex projective spaces. His very recent attempts on the Feit-Thompson theorem related to this idea.

Two years ago  there were  meetings  at breakfast  several  mornings in Shanghai at the Alain Connes fest to  critically discuss his reasoning why $S^6$  should  not have an  integrable complex structure. Sometimes, he left the table  grumbling, but  returned to  battle the next day. The discussions did not resolve the issue.

With  Michael Atiyah's   strong intuition and deep  insight across many fields,   it seemed important  then,  and  it seems important now, to leave no stone unturned in  pursuing his ideas. We will miss Sir Michael Atiyah's unfettered enthusiasm for deep and broad mathematics and his exciting ideas for its advance.

\end{document}